\documentclass[12pt,
]{amsart}

\usepackage{hyperref,
amssymb, graphicx,  wrapfig, cite, color
}
\usepackage[font=small,labelfont=bf]{caption} 
\usepackage[subrefformat=parens]{subcaption}
\usepackage[utf8]{inputenc}
 
\newcommand{\mar}[1]{{\marginpar{\sffamily{\scriptsize #1}}}} 

\date{\today}

\numberwithin{equation}{section}%

\newtheorem{theorem}{Theorem}[section]
\newtheorem{proposition}{Proposition}[section]
\newtheorem{lemma}{Lemma}[section]
\newtheorem{definition}{Definition}[section]

\theoremstyle{plain}
\newtheorem{assumption}{Assumption}

\theoremstyle{definition}
\newtheorem{remark}{Remark}[section]

\DeclareMathOperator{\erf}{erf}
\DeclareMathOperator{\sinc}{sinc}

\DeclareMathOperator{\sgn}{sgn}

\DeclareMathOperator{\supp}{supp}

\DeclareMathOperator{\pv}{pv}

\DeclareMathOperator{\WF}{WF}

\DeclareMathOperator{\WFH}{WF_{\it h}}

\renewcommand{\sc}{semi\-classical}
\newcommand{\R}{{\mathbb R}}
\newcommand{\RR}{{\mathcal{R}}}

\renewcommand{\r}[1]{\eqref{#1}}
\renewcommand{\Xi}{B}
\newcommand{\PDO}{$\Psi$DO}
\newcommand{\HPDO}{$h$-$\Psi$DO}
\newcommand{\be}[1]{\begin{equation}\label{#1}}
\newcommand{\ee}{\end{equation}}

\renewcommand{\d}{\mathrm{d}}

\renewcommand{\i}{\mathrm{i}}

\title[The Radon transform with finitely many angles]{The Radon transform with finitely many angles}

\author[Plamen Stefanov]{Plamen Stefanov}
\address{Department of Mathematics, Purdue University, West Lafayette, IN 47907}

\thanks{P.S.\ partially supported by the National Science Foundation under
grant DMS-1900475.}

\begin{document}

\begin{abstract}
We study the Radon transform in the plane in parallel geometry possibly undersampled in the angular variables. We study resolution, aliasing artifacts, and edge recovery. 
\end{abstract} 
\maketitle
\setcounter{tocdepth}{1}
\tableofcontents

\section{Introduction} The purpose of this paper is to study the Radon transform $\RR f(\omega,p)$ in ``parallel geometry'', see \r{T1}, in the plane with discrete measurements. We assume that the measurements  $\RR f(\omega,p)$ are well sampled in the variable $p$ but undersampled in the angular variable $\omega$. This corresponds to practical situations where the measurements are taken at finitely many angles but not as many as needed for good resolution and to avoid aliasing; on the other hand, at each angle, the image is resolved well by a high enough resolution device. We are interested in describing the aliasing artifacts, the resolution limit, and in particular, recovery of edges and jump type singularities.

Sampling $\RR f$ pointwise for $f\in L^\infty_{\rm comp}$ is not a well posed problem since $\RR f$ does not have enough regularity to have well-defined pointwise values, even if $f$ is piecewise smooth. On the other hand, discrete measurements in practice are not done pointwise (even if $\RR$ is not a Radon transform) since pointwise, we would measure zero signal with significant noise. Typically, they are locally averaged. In the case of the Radon transform, the X-rays are not ideal rays; they are either collimated or issued from a very small source, and additionally blurred by diffraction, see also \cite{Cormack_78}. When the X-ray projection (as a function of the $p$ variable) is taken at each fixed angle, it is averaged over small detectors (pixels). On the other hand, the finite number of angles/directions of those projections could be modeled as pointwise measurements of an already locally averaged signal. That averaging can be passed to $f$ by Egorov's theorem, which allows us to think of pointwise measurements in the angular variable (not locally averaged) of a slightly blurred copy of $f$. 
To make things simple, as mentioned above, we assume  high enough resolution in the $p$ variable at each fixed angle so that we can assume formally that we have a function known for all $p$; and this can be justified by the sampling theory.  

The resulting recovery depends on the way it is done even if we just want to apply the filtered backprojection in a discrete setting. We compare two ``natural'' implementations of that formula, and show that  they produce different results, in particular each one produces aliasing artifacts, as expected, but they are different. We analyze the method we call ``direct'' in more detail since this is the commonly used one. The other one, which we call the ``interpolation'' method was already analyzed in \cite{S-Sampling}, and it turns out to produce a reconstruction which  is an angularly averaged version of the direct one, see Theorem~\ref{thm_conv}, making it of less interest, probably. 

We analyze the problem both with ``classical''  and \sc\ (asymptotic) methods. The classical point of view is:  how well or not classical singularities are resolved. The most general tool for that would be FIOs associated with a pair of cleanly intersecting Lagrangians, we refer to Remark~\ref{remark1}(a). More direct methods studying singularities added by a singular cutoff applied to the data, see, e.g., \cite{Frikel-Quinto, BorgFJSQ, Linh-15} can be used as well, see also Theorem~\ref{thm_WF} below.
We do not do full analysis --- we just study edge recovery, a partial case of recovery of conormal singularities.

The \sc\ (asymptotic) analysis follows in parts the theory developed by the author in \cite{S-Sampling}: an asymptotic sampling theory as the sampling step tends to zero for (linear) Fourier Integral Operators (FIOs) with a canonical relation being locally the graph of a map. The Radon transform is a particular example, and the approach has been applied to Thermoacoustic Tomography as well \cite{Chase_Sampling} and to the geodesic X-ray transform \cite{MS-sampling}. We assume that the sampling step is proportional to a small parameter $h>0$, and used the \sc\  pseudodifferential and FIO calculus. Then $\RR$ acts on functions depending on $h$ as well, oscillating highly but still smooth. This is one of the technical tools used in this paper. Using it, one can handle undersampling in $p$ as well, as in \cite{S-Sampling}. 

We want to emphasize that in sampling theory, the reconstruction from samples depends on the way the interpolation is done, naturally. It could be the Whittaker–Shannon interpolation formula ($\sinc$ based) or some version of it if there is oversampling, or even, say linear/bilinear, cubic interpolation, etc. The method we call ``direct'' has no interpolation involved, and yet,  sampling theory appears naturally through the Poisson summation formula, see section~\ref{sec_3}. 

The \sc\ treatment has the following advantages. Besides modeling dense enough measurements, it is also useful in  numerical computations, when the small parameter $h$ is proportional to  the step size (when using a mesh). 
Next, classical microlocal analysis is asymptotic in the sense that it cares about the Taylor-like expansion of the Fourier transform at the infinite sphere $|\xi|=\infty$. Roughly speaking, it misses what happens on the way there. An oscillating function, like $\cos(k x)$ with $k\to\infty$, for example, is smooth, thus negligible in  classical microlocal sense. In a \sc\ sense, it has \sc\ singularities, and it is not an approximate classical singularity in any reasonable sense; in fact, its weak limit is zero, as $k\to\infty$. 

It is known that $\RR$, restricted to finitely many directions, has a non-trivial kernel, see, e.g., \cite{Helgason-Radon}. In \cite{Louis-81}, Louis studies the ``ghosts'', i.e., the null-space. In \cite{Louis-84}, he describes the ghosts in all dimensions as a high-frequency phenomenon, generalizing previous works. This is close in spirit to our asymptotic approach but the methods and the conclusions we get are of a very different nature. Sampling for the Radon transform has been studied in \cite{Cormack_78, Natterer-book, Natterer-sampling1993,Rattey-sampling}, and more recently in \cite{S-Sampling, MS-sampling}, and by  Katsevich 
\cite{Katsevich_17, Katsevich2019analysis, alex2020analysis, Katsevich2021}. His approach is different from ours, and the conclusions cannot be compared directly. This work was inspired in part by a conversation by the author and Katsevich. The author thanks Fran\c{c}ois Monard for the discussions and for the references  \cite{Louis-81, Louis-84}.

\section{Preliminaries}
\subsection{The filtered backprojection} 
We work in the plane. 
The  Radon transform is defined by 
\be{T1}
\mathcal{R}  f(\omega,p) = \int_{x\cdot \omega =p}  f(x)\,\d\ell,
\ee
where   $\d\ell$ is the Euclidean length measure, and $\omega\in S^1$, say parameterized as 
\be{omega}
\omega(\varphi): = (\cos\varphi,\sin\varphi). 
\ee
We will denote by $\omega^\perp = (-\omega_2,\omega_1)$ its  rotation by $\pi/2$. 
We always think that $\varphi\in [0,2\pi]$ is  as a parameterization of the circle $S^1$, i.e., identifying $0$ and $2\pi$. More generally, we assume $\varphi\in \R/2\pi\mathbb Z$.  
The Radon transform is even, i.e., it is invariant under the map $(\omega,p)\mapsto (-\omega,-p)$, i.e., $(\varphi,p) \mapsto (\varphi+\pi, -p)$. When we study the microlocal properties of $\RR f$, we think of it as a function of $(\varphi,p)$. 

A popular inversion formula is the so-called filtered backprojection
\be{FBP}
f = \RR'\mathcal{H}g, \quad g=\RR g,
\ee
where $\mathcal{H} = \frac1{4\pi} H d_p$, with $d_p=\partial/\partial p$, and $ H$ being the Hilbert transform
\be{H}
H g(p)=\frac1\pi \pv \int \frac{g(s)}{p-s}\d s.
\ee
One of the advantages of this formula is that if $g=\RR f$ with $f$ compactly supported, then so is $g$; and to compute the inversion for $x$ in a compact set for $x$, we need to compute $Hd_p$ with $p$ and $s$ over a bounded interval (for every $\varphi$) only. We note that $H$ is the Fourier multiplier by $-\i \sgn(\hat p)$, therefore $Hd_p = |D_p|$. We denote by $\hat\varphi$ and $\hat p$ the  variables dual to $\varphi$ and $p$, respectively.  

\subsection{Discrete data} 
Assume we are given the Radon transform $\RR f(\omega,p)$ sampled on a (finite) discrete set of points $\{\omega_i, p_j\}$. We always assume that $\supp f\subset \mathcal{B}(0,R)$ with $R>0$ fixed, where $\mathcal{B}(0,R)$ is the ball with center $0$ and radius $R$.

We consider the following two methods of applying the filtered backprojection \r{FBP} given discrete data. The first one that we call the \textit{interpolation method} is to interpolate $\RR f(\omega_j, p_j)$ to get a function for all ``continuous'' $(\omega,p)$, and then apply \r{FBP}. This can be done approximately on a finer grid. The second one, which we call the \textit{direct method} is to apply $\mathcal{H}$, and then $\RR'$ using discrete approximations of each of those operators. We assume that there is oversampling in the $p$ variable, which allows us to recover $\RR f(\omega_j, p)$ for all $p$ with a small error. Then we can apply $\mathcal H$ to it. The direct method then is to perform numerical integration  by summing up $(\mathcal{H} \RR f)(\omega_j, x\cdot\omega_j)$ (i.e., replace the actual integral with Riemann sums \r{f_delta}), while the interpolation method interpolates $(\mathcal{H} \RR f)(\omega_j,p)$ first to $(\mathcal{H} \RR f)(\omega, p) $ and only \r{P2}, then sets $p= x\cdot\omega$ and integrates. 

To have the flexibility to consider the limited angle problem, let $\psi\in C^\infty(S^1)$ be a cut-off function, and assume we are given 
\be{data}
g_j(p) = \psi (\omega_j )\RR f(\omega_j,p).
\ee
 For simplicity, we assume that $\psi$ is even. In fact, since   $\RR$ is even, we can always symmetrize $\psi \RR f$, so this assumption is not restrictive. Then we replace $\RR f$ above by  $\psi \RR f$. 

\subsection{The asymptotic approach} 

In \cite{S-Sampling}, we propose an asymptotic point of view. We give more details in appendix~\ref{sec_appendix}. Say that the sampling rates are proportional to a small parameter $h>0$, and we want to understand the asymptotic behavior as $h\to0$. In other words, the sampling rates are $hs_\varphi$ and $hs_p$, and we call $s_\varphi$ and $s_p$ \textit{relative}  sampling rates. Ignoring possible offsets relative to the origin, we can assume 
\be{phi}
\varphi_i = i h s_\varphi, \quad p_j = jhs_p,\quad  \text{and denote $\omega_i=\omega(\varphi_i)$}. 
\ee
Since $\varphi$ parameterizes the unit circle, one has to worry about periodicity or not of the sequence $\varphi_j$.  
We assume:

\begin{assumption}\label{assumption}
The number $\pi/hs_\varphi =:m$ is an integer.
\end{assumption}

This restricts $h$ to the set $h\in \{\pi/s_\varphi m;\;m\in\mathbb{N}\}$ for any $s_\varphi>0$ fixed. Then the number of the distinct $\omega_j$ is equal to $2m$. That set is even, and since $\omega_j$ and $-\omega_j$ define the same families of lines parallel to each one of those directions, we actually have $m$ distinct families of parallel lines. Next, $\RR f$ is even, and so is $\psi$, so we can work with half of those $\omega_j$'s (so that adding the opposite ones completes the whole set), as it is usually done. 

Part of our analysis is not asymptotic, then one can take $h$ fixed, say $h=1$. Then $\pi/s_\varphi=m$ is an integer. When we do an asymptotic analysis, we take $h\to0$, which is to say that $m$ is a large parameter. 

In the asymptotic part, we work with functions $f(x)$ depending on $h$ as well, \sc ly band limited in the ball $|\xi|\le B$, see  Appendix. To motivate the interest in such functions, fix a function $\psi$ so that $\hat\psi\in C_0^\infty(\R)$, and set $\psi_h=h^{-1}\psi(\cdot/h)$. In practice, $\hat\psi$ may not be of compact support but can decay fast enough to be considered such with a small error. The locally averaged measurements are then modeled by $\psi_h *_p \mathcal R_\kappa$, where $*_p$ is the convolution with respect to the $p$ variable.  Egorov's theorem implies  $\psi_h *_p \mathcal R f(x) =\mathcal{R}Q_h f +O(h^\infty)$, where $Q_h$ is an \HPDO\ away from $\xi=0$ with principal symbol $\psi_h(|\xi|)$. This can be made more explicit with the use of the well-known intertwining property of the Radon transform. This observation has two implications: (1) if $f$ is, say, $L^\infty$ only, and independent of $h$, then $\tilde f_h:= Q_h f$ is $h$-dependent and \sc ly band limited; and (2) we can replace averaged measurements near a discrete set of points by pointwise measurements of $\mathcal{R}\tilde f$.

\section{The direct method, classical (non-asymptotic) view} \label{sec_3}
The method we call ``direct'' consists of the following. We take $h=1$ in this section. The asymptotic analysis as $h\to0$ will be done in the next one. Also, the step size $s_\varphi$ is simply denoted by $s$, so $\varphi_j = sj=\pi j/m$. 
Given the discrete data $\RR f(\omega_j,p)$, we compute $\mathcal{H}\RR f(\omega_j,p)$ as before, which is an operation in the $p$ variable for any $\omega$ fixed. If we knew  $\mathcal{H}\RR f$ for all $\omega$ (and $p$, of course), the inversion would have been 
\be{D1}
f(x) = \int_{S^1} \mathcal{H}\RR f(\omega , x\cdot\omega )\, \d\omega,
\ee
which is just \r{FBP}.  Instead, we perform numerical integration with the given 
\begin{figure}[h!] 
\includegraphics[height=.16\textheight, trim={0 25 0 20},clip]{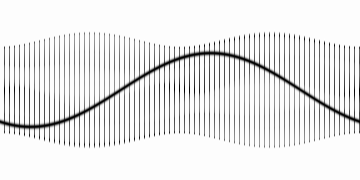}
\caption{The sinogram of the Shepp-Logan phantom with step size $5^\circ$. One first applies $\mathcal{H}$ in the $p$ variable, and then for each $x$, one computes the integral  $\int \mathcal{H}\RR f(\varphi, x\cdot\omega(\varphi))\, \d\varphi$ numerically by summing up the values at each angle along the curve representing the lines through $x$.}\label{fig-SL}
\end{figure}
samples by summing up over $\omega_j$ (and multiplying by the step size $s=\pi/m $) to get
\be{f_delta}
f_\delta(x) := \frac{2\pi}{m}\sum_{j=1}^{m} \mathcal{H}\RR f(\omega_j, x\cdot\omega_j). 
\ee
Note that first, we would get a sum from $j=1$ to $2m$ with the coefficient $\pi/m$ in front. Since $\mathcal{H}\RR f$ is even, by Assumption~\ref{assumption}, we can reduce the summation as indicated and multiply by $2$. 
 This is what \texttt{iradon} in the current version of   MATLAB does, for example.

Consider limited angle data $g= \psi\RR f$ now. 
If we had the non-discretized data, the natural inversion would have been 
\be{D1p}
f_\psi(x) : =  \int_{S^1} \mathcal{H} g(\omega,x\cdot\omega) \,\d\omega=  \int_{S^1} \psi(\omega)\mathcal{H} \RR f(\omega , x\cdot\omega )\, \d\omega,
\ee
(note that $\mathcal{H}$ commutes with $\psi$), 
and with discrete data, we do numerical integration 
\be{f_delta_d}
f_{\psi,\delta} (x) := s \sum_j (\mathcal{H}g_j)(x\cdot\omega_j)   = \frac{2\pi}{m} \sum_{j=1}^m \psi (\omega_j) (\mathcal{H}\RR f)(\omega_j, x\cdot\omega_j)
\ee
as in \r{f_delta} instead. The subscript $\delta$ can be explained by formula \r{P2d} below. 
We are not claiming that \r{D1p} is the ``best'' inversion with limited angle data; in fact this is a problem with a lot of proposed ``solutions'' (and without a unique solution \cite{Helgason-Radon}). It recovers the singularities stably recoverable from the data however. It follows from \cite{S-identification}, for example, that
\be{D1pa}
f_\psi =  \psi (D/|D|)f.
\ee

The following theorem follows easily from the calculus of wave front sets and the explicit form of the canonical relation of $\RR$.

\begin{theorem}\label{thm_WF}
Let $f\in \mathcal{E}'(\R^2)$. Then $\WF(f_{\psi,\delta})$ is included in the conormals of all lines $\{x;\; (x-x_0)\cdot\omega_j=0\}$ whenever $(x_0,\omega_j)\in \WF(f)$ for some $j$.
\end{theorem}

In other words,  all singularities of $f_{\psi,\delta}$ are included in the following set: for every $(x_0,\omega_j)\in \WF(f)$ we take the  conormals to the line through $x_0$ and conormal to $\omega_j$. The examples we present below show that in some cases, this inclusion is actually an equality.   

\begin{proof}[Proof of Theorem~\ref{thm_WF}]
Set 
\be{P2d}
(\psi \RR f)_{{\rm int},\delta}(\omega,p) = \frac{2\pi}{m}\sum_{j=1}^m \psi( \omega_j  )\RR f(\omega_j ,p) \delta(\omega-\omega_j),
\ee
which is as \r{P2} but with $\chi=\delta$ there. The Riemann sum \r{f_delta_d} is an actual integral of $(\psi \RR f)_{{\rm int},\delta}(\omega, x\cdot\omega)$ in the $\omega$ variable, i.e.,  $f_{\psi,\delta} = R'(\psi \RR f)_{{\rm int},\delta}$. By the calculus of the wave front sets, see, e.g., \cite{Hormander1}, the wave front of the product $\RR f(\omega ,p) \delta(\omega-\omega_j)$ is the closure of the vector sum of the wave front set of each factor (the delta considered as a distribution w.r.t.\ $(\omega,p)$). Note that the product is well-defined since $\WF(\RR f)$ is separated from  the conormals $\lambda \d \varphi$ for $f$ compactly supported, as it follows easily from the expression for the canonical relation of $\RR$, see, e.g., \cite{S-Sampling}. Then the closure of that vector sum is the whole $\R^2\setminus 0$ over every point where $\RR(\omega_j,p)$ is singular. Then all those covectors over any such point would be mapped to singularities conormal to the line $\omega_j\cdot x=p$ by the inverse canonical relation. 
\end{proof}

Even if $f$ is piecewise smooth, $f_{\psi,\delta}$ might be a distribution, not a function, see section~\ref{sec-classical}.

\begin{theorem}\label{thm_cl}
Assume $f\in\mathcal{S}(\R^2)$, and let $\psi\in C^\infty(S^1)$ be even. 
Let Assumption~\ref{assumption} hold. Let $f_{\psi,\delta}$, given by \r{f_delta_d}, be the reconstructed $f$ with discrete limited angle data $g_j$ given by \r{data}.  Then

(a) $f \mapsto f_{\psi,\delta}$ is the Fourier multiplier
\be{S7}
\hat f_{\psi,\delta}(\xi) = \frac{\pi}{m} \sum_{j=1}^m \psi(\omega_j)  \delta\left(\xi\cdot\omega_j^\perp\right)  |\xi|  \hat  f (\xi).
\ee
Also, %
\be{S7'}
\hat f_{\psi,\delta}(\xi) = \frac{\pi}{m} \sum_{j=1}^m  \delta\left(\xi\cdot\omega_j^\perp\right)  |\xi|  \Big[\psi(\xi/|\xi|) \hat  f (\xi)\Big]
\ee
thus, $f_{\psi,\delta}$ is a linear operator applied to $f_\psi$, which is a Fourier multiplier as well. 

(b) We also have
\be{S8}
f_{\psi,\delta} = f_\psi + \sum_{k=1}^\infty G_k f_\psi,
\ee
where $f_\psi$ is as in \r{D1pa}, and  $G_k$ are the Fourier multipliers
\be{S8a}
G_k : \hat f(\xi)  \longmapsto 
2  \cos\big( 2m k\arg(\xi) \big)  \, \hat f(\xi).
\ee
\end{theorem}
\begin{proof}
By the Fourier Slice Theorem,
\[
\mathcal{F}_{p\to\hat p} \RR f(\omega,p) = \int e^{-\i \hat p p} \RR f(\omega ,p) \, \d p = \int e^{-\i \hat p y\cdot\omega} f(y)\,\d y = \hat f(\hat p\omega). 
\]
Then
\[
\mathcal{F}_{p\to\hat p} \psi \mathcal{H}\RR f(\omega ,p) = \frac1{4\pi}|\hat p| \psi(\omega) \hat f(\hat p\omega).
\]
Therefore,
\[
 \psi(\omega)\mathcal{H}  \RR f(\omega ,p) = \frac1{8\pi^2} \psi(\omega)\int e^{\i \hat p p} |\hat p|\hat f(\hat p\omega)\,\d \hat p.
\]
Hence,
\be{S6a}
(\psi   \mathcal{H} \RR f)(\omega ,x\cdot\omega) = \frac1{8\pi^2}\psi(\omega)  \int e^{\i \hat p x\cdot\omega} |\hat p|  \hat f(\hat p\omega) \d \hat p.
\ee

We apply the definition \r{f_delta_d} of $f_{\psi,\delta}$ now: we discretize \r{S6a} to plug it in \r{f_delta_d}, and then take the Fourier transform:
\[
\begin{split}
\hat f_{\psi, \delta}(\xi) &= \frac{1}{4\pi m}\sum_{j=1}^m \psi(\omega_j)  \iint e^{-\i  x\cdot\xi+ \i \hat p x \cdot\omega_j} |\hat p|  \hat f(\hat p\omega_j) \,\d \hat p\,\d x\\ 
&= \frac{\pi}{m}\sum_{j=1}^m \psi(\omega_j)  \int  \delta(\xi-\hat p\omega_j )|\hat p| \hat f(\hat p\omega_j)  \,\d \hat p. 
\end{split}
\]
Split the integral above into one over $[0,\infty)$ and the other one over $[-\infty,0]$, make the change of variables $\hat p\mapsto -\hat p$, and shift the index $j$ so that  $\{\omega_j\}$ gets multiplied by $-1$ after this. As a result, $\{\omega_j\}_{j=m+1}^{2m}$ are brought up into the sum, and in the second integral, $\hat p\in [0,\infty)$. 
In other words, we can extend the summation to $j=1,\dots,2m$ but restrict the integration to $\hat p>0$ only above. 

Given a test function $\rho$, we have
\[
\begin{split}
\langle \hat f_{\psi, \delta}, \rho\rangle &=  \frac{\pi}{m} \sum_{j=1}^{2m} \psi(\omega_j) \int_0^\infty |\hat p| \hat f(\hat p\omega_j)\rho(\hat p\omega_j )  \,\d \hat p\\
& = \frac{\pi}{2m} \sum_{j=1}^{2m} \int_{S^1} \int_0^\infty\hat f(\hat p\omega)\rho(\hat p\omega ) \delta(\omega\cdot\omega_j^\perp)  |\hat p| \,\d \hat p\,\d\omega\\
& = \frac{\pi}{2m} \sum_{j=1}^{2m}\psi(\omega_j)\int  \hat f(\xi)\rho(\xi ) \delta((\xi/|\xi|) \cdot\omega_j^\perp)    \,\d \xi.
\end{split}
\]
Therefore, 
\[
\hat f_{\psi, \delta}(\xi)= \frac{\pi}{2m}\sum_{j=1}^{2m} \psi(\omega_j) |\xi| \delta\big(\xi\cdot \omega_j^\perp\big)      \hat f(\xi),
\]
which can be written as \r{S7} as well. 
To gain  reader's confidence about this computation, assume $\psi=1$, and note that as $s=\pi/m\to 0$,  the number of samples $2 m$ on the circle is increasing, and the formula above converges to  $\frac12\int |\xi|\delta(\xi\cdot\omega^\perp)\,\d\omega=1$, multiplied by $\hat f(\xi)$, as one would expect. This  proves  \r{S7} in (a). Next, note that on the support of $\delta(\xi\cdot\omega_j^\perp)$, we have $\omega_j=\pm \xi/|\xi|$. Since $\psi$ is even, this proves \r{S7'}, and completes the proof of (a). 

To prove (b), first we want to connect the actual $\omega$ integral in \r{D1p} with its Riemann sum in \r{f_delta_d} in a Fourier transform kind of way. 
We consider the one-dimensional version first. The Poisson summation formula implies, say for $\rho$ in the Schwartz class,
\be{S4}
s \sum_{k\in \mathbb Z} \rho(sk) =  \sum_{k\in \mathbb Z} \hat \rho(2\pi k/s),
\ee
where $s>0$ is fixed, and in the sequel, $s=s_\varphi = \pi/m$, as above.  
Note that if $\rho$ is a  \textit{classically} band-limited function with frequencies in $[-B,B]$, then if $\pi/s>B$ (the Nyquist condition), only the $k=0$ term on the right in \r{S4} would be possibly different than zero. Then \r{S4} can be interpreted as saying that the Riemann sum on the left, approximating $\int \rho(x)\,\d x = \hat \rho(0)$ is exact for such functions.  When the Nyquist condition is not satisfied, then \r{S4} is exact for the integral of the aliased reconstruction of $\rho$, and the 
 $k\not=0$ terms on the right represent corrections coming from the aliased components. 
 
 In our case $\rho$ is a $2\pi$ periodic function, and we sum over $k\in \mathbb{Z}/2m\mathbb{Z}$ 
(say, over $k\in \{-m+1,\dots,m\}$). We can view $\rho$ as supported on $[0,2\pi]$, then extended as zero outside that interval for the purpose of the summation. Then in \r{S4}, $\hat\rho$ is evaluated at $2mk$, i.e., those are just Fourier coefficients of $\rho$. With this in mind, we write
\be{S4a}
\frac{\pi}m \sum_{k\in \mathbb{Z}/2m\mathbb{Z} } \rho(\pi k/m) =  \sum_{k\in \mathbb{Z} } \hat \rho(2mk).
\ee
 
We  apply \r{S4a}  to \r{f_delta_d} with $\rho$ being the function  $\varphi \mapsto (\psi \mathcal{H}\RR f)(\omega(\varphi), x\cdot\omega(\varphi))$. We get
\be{S5}
f_{\psi, \delta}  = \sum_k f_{\chi, \delta}^{(k)}, \quad f_{\chi, \delta}^{(k)}:= \mathcal{F}_{\varphi\to\hat\varphi} (\psi  \mathcal{H}\RR f)(\omega(\varphi), x\cdot\omega(\varphi))\big|_{\hat\varphi=2mk}.
\ee

By  \r{S6a}, each aliased component in \r{S5} is then given by
\[
f_{\psi, \delta}^{(k)} (x) = \frac1{8\pi^2}\iint \psi(\omega(\varphi)) e^{-\i 2 m k\varphi }  e^{\i \hat p x \cdot\omega(\varphi)} |\hat p| \hat f(\hat p\omega(\varphi))\, \d \hat p\,\d\varphi.
\]
We split the $\hat p$  integration in two parts: over $\hat p>0$ and over $\hat p<0$. In the second one, we make the change $(\varphi,\hat p)\mapsto (\varphi+\pi, -\hat p)$ to get
\[
\begin{split}
f_{\psi, \delta}^{(k)}(x)  & = \frac1{8\pi^2 }\int  e^{ \i  x\cdot\xi}\psi(\xi/ |\xi| ) e^{-\i 2m k\arg(\xi) }   \hat f(\xi) \,\d \xi\\
&\quad  + \frac1{8\pi^2 }\int  e^{ \i x \cdot\xi} \psi(-\xi/|\xi|) e^{-\i 2m k\arg(-\xi) }  \hat  f(\xi) \,\d \xi\\
& = \frac1{4\pi^2 }\int  e^{ \i  x\cdot\xi}\psi(\xi/ |\xi| ) e^{-\i 2m k\arg(\xi) }   \hat f(\xi) \,\d \xi
\end{split}
\]
Then
\be{S6b}
f_{\psi, \delta}^{(k)} (x) + f_{\psi, \delta}^{(-k)} (x)  = \frac1{(2\pi)^2 }\int e^{ \i x\cdot\xi}\psi(\xi/|\xi|)2 \cos( 2 m k\arg(\xi) )  \hat  f(\xi)\,\d \xi.
\ee
This completes the proof of (b). 
\end{proof}

\begin{remark} \label{remark1}  \ 

(a) By \r{S7}, the map $f_\psi \mapsto f_{\psi,\delta}$ is a formal \PDO\ but with a singular symbol. Such operators are studied in  \cite{Melrose-Uhlmann-79, Guillemin-Uhlmann-81, Antoniano-Uhlmann}. 
This allows for a point of view more general than that of Theorem~\ref{thm_WF} and that in section~\ref{sec-classical}. 
Also, \r{S7} can be considered as $\RR' \RR_b$ with $b$ a singular weight, and the formula is the same when $b$ is smooth, see the appendix in \cite{S-identification}.

(b) 
Each $G_k$ in \r{S8} is a \PDO\ of order zero, and as such, it does not add additional singularities. The infinite sum however, may, in general, as it is seen from \r{S7}, see also section~\ref{sec-classical}.

(c) 
When we view $G_k$ asymptotically, as $m\to\infty$, then $G_k$ are interpreted as \sc\ FIOs, when $m$ is considered as a large parameter; which add, and also displace (\sc) singularities. 


(d) Formulas \r{S7} and \r{S8} can be obtained from each other by the Poisson summation formula, as the proof shows. 
\end{remark}

\begin{remark}
The operator $G_k$ is a convolution with 
\[
C_{mk} |x|^{-2}\cos(2mk\arg(x)), 
\]
see \cite{Lemoine}. The singularity at $x=0$ is in principal value sense  since the cosine function there has a zero mean value over the unit circle. 
\end{remark}

\section{The direct method, an asymptotic view} 
\subsection{The aliasing as a \sc\ FIO}
We take the asymptotic view now: the angular step size is $sh$ now with $s>0$ fixed and $h\to0+$. As in Theorem~\ref{thm_cl}(b) above, according to Assumption~\ref{assumption}, we assume that $sh=\pi/m$ with $m\in\mathbb N$. Then $h\in\{ \pi/ms;\; m\in\mathbb{N}\}$ when $s>0$ is fixed, and our analysis is asymptotic, as $m\to\infty$. Formula \r{f_delta_d} takes the form
\be{f_delta_d'}
f_{\psi,\delta} (x) := 2sh \sum_{j=1}^{\pi/hs} \psi(\omega_j) (\mathcal{H}\RR f)(\omega_j, x\cdot\omega_j), \quad \omega_j = \omega(shj). 
\ee
The function $f$ is assumed to be $h$-dependent, and \sc ly band limited, say with $\WFH(f)\subset \mathcal{B}(0,R)\times \mathcal{B}(0,B)$.  Theorem~\ref{thm_cl} still holds but we replace $\hat f$ now by its \sc\ version $\mathcal{F}_hf(\xi) = \hat f(\xi/h)$. The relevant part is (b) in this case. Then, with $f_\psi$ still given by \r{D1pa}, it follows from \r{S8a} that
\be{DA1}
G_k : \mathcal{F}_h f(\xi)  \longmapsto 
 2  \cos\frac{ 2\pi k\arg(\xi)}{sh}   \mathcal{F}_h f(\xi).
\ee
It is convenient to write the cosine function as a sum of complex exponentials, the way we derived it in the first place:
\be{DA2}
G_k = \mathcal{A}_k+ \mathcal{A}_{-k}, \quad \mathcal{A}_{\pm k}: \mathcal{F}_h f(\xi)  \longmapsto  e^{ \pm 2\pi\i  k\arg(\xi)/sh} \mathcal{F}_h f(\xi) , \quad k=1,2,\dots. 
\ee
We get the sum of two unitary FIOs away from the zero section. The phase functions are $ \Phi_\pm(x,y,\xi) = \pm 2\pi k\arg(\xi)/s  + (x-y)\cdot\xi $. The characteristic variety $\Sigma:= \{\Phi_\xi=0\}$ given by 
\[
y=x\pm \frac{2\pi k}{s} \xi^\perp/|\xi|^2.
\]
Then  $(x,\Phi_x)\mapsto (y,-\Phi_y)$ on $\Sigma$ is actually a graph, so we get that the canonical relations $C_k$ of $\mathcal{A}_k$, where $k$ can be negative as well, are
\be{CR}
C_{k} :(x,\xi) \mapsto \Big(x+ \frac{2\pi k}  {s|\xi|} \frac{\xi^\perp}{|\xi|},\xi\Big).
\ee
This leads to the following. 

\begin{theorem}\label{thm_1}
Let $f_h$ be \sc ly band limited. 
Then, given its discretized localized Radon transform $g_j$ as in \r{data} at $\varphi_j=shj$, the reconstructed $f_{\psi, \delta}$ by \r{f_delta_d} has the form
\be{thm_1_eq1}
f_{\psi, \delta} = f_\psi+ \sum_{k=-\infty, \,k\not=0}^\infty \mathcal{A}_kf_\psi, 
\ee
where $\mathcal{A}_k$ are the Fourier multipliers given by \r{DA2}. 
Also, $\mathcal{A}_k$ are unitary, and  away from the zero section, they are elliptic \sc\ FIOs of order zero  with canonical relations $C_{k}$ given by  \r{CR}.
\end{theorem}

\begin{remark}\label{remark_4.1}
Theorem~\ref{thm_1} shows that while artifacts are always created, the original $f$ appears in the expansion as well. In this sense, no resolution has been lost!  That term $f$ could, in principle overlap or even be canceled by the artifacts of another singularity elsewhere.
To avoid aliasing artifacts in a fixed ball $\mathcal{B}(0,R)$, we need $2\pi /(s|\xi|)>2R$, i.e., $s<\pi/BR$. This is the same requirement we got in \cite{S-Sampling}, see also \r{F1}. This is formulated in Theorem~\ref{thm_art}(b) below.  
\end{remark}

\begin{remark}\label{rem_4.2}
 In particular, we get that the sum in \r{thm_1_eq1} is locally finite.  Indeed, since $|\xi|\le B$ by the assumption on $f$, we have $2\pi/s|\xi|\ge 2\pi/sB$, therefore $C_k$ shifts each $x$ at least at distance $2\pi k/sB$, and for $k\ge  k_0\gg1$. Then this would leave any fixed compact domain since $|x|\le R$, and that $k_0$ depends on $R$ and $B$ only. We also get a lower bound, $2\pi/sB$, of the distance of the artifacts to $x$. 
\end{remark}


\begin{theorem}\label{thm_art}
Assume  that $f=f_h$ is a \sc ly band limited function with $\WFH(f)\subset \mathcal{B}(0,R)\times\mathcal{B}(0,B) $ for some $R>0$, $B>0$. 
Then

(a) 
\[
\WFH (f_{\psi, \delta} )\setminus 0\subset \bigcup_{k\in \mathbb Z}C_k(\WFH(f_\psi )\setminus 0). 
\]

(b) If  $B<\pi/sR$, then $f_{\psi,\delta}= f_\psi + O_{H^s}(h^\infty)$ in $\mathcal{B}(0,R)$, $\forall s$.  

(c) Under the condition of (b), 
\be{art1}
\begin{split}
\WFH (f_{\psi, \delta} - f_\psi)\setminus 0 &= \bigcup_{k\in \mathbb Z\setminus 0}C_k(\WFH(f_\psi )\setminus 0).
\end{split}
\ee
\end{theorem}

\begin{proof}
Part (a) follows directly from the properties of \sc\ FIOs, see \cite{Martinez_book}, \cite{Guillemin-SC}. Part (b) follows from (a) since  for $k\not=0$, $C_k$ sends $\WFH(f)$ outside $T^*\mathcal{B}(0,R)$ under the condition assumed. 

For part (c), note first that the inclusion $\subset$ follows as in (a). To prove the equality, choose $(x^\sharp,\xi^\sharp)$ in the union on the right, say corresponding to $k=k_0$, i.e., $(x^\sharp,\xi^\sharp)= C_{k_0}(x_0,\xi_0)$ for some $(x_0,\xi_0)\in \WFH(f_\psi)$. Then $k_0$ and $(x_0,\xi_0)$ with that property are uniquely determined. Indeed, we must have $\xi_0=\xi^\sharp$ (we work in a fixed coordinate system, and comparing covectors at different points makes sense); and there is unique $k$ so that $C_k^{-1}(x^\sharp,\xi_0) = C_{-k}(x^\sharp,\xi_0)$, see \r{CR}, would land in $\mathcal{B}(0,R)$ since $B<\pi/sR$. This is also true under a small perturbation of $(x^\sharp,\xi^\sharp)$. 
Next, each $\mathcal{A}_k$ is elliptic, which proves the claim.
\end{proof}

\begin{remark}
We want to emphasize that the equivalent to (b) above in  \cite{S-Sampling} was derived about the interpolation method we discuss in next section, i.e., we have the same about $f_{\psi,\chi}$ we study there. Then the two methods are equivalent when the Nyquist condition holds. 
\end{remark}

\begin{figure}[h!] 
  \centering
	\includegraphics[page=2, scale=1]{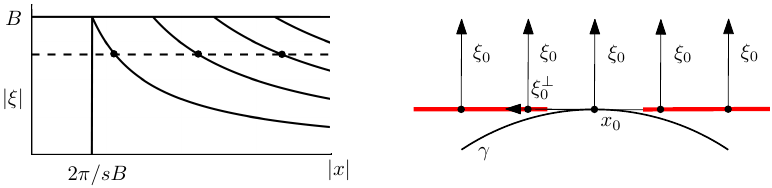}
\caption{\small Illustration to the canonical maps $C_k$ of $\mathcal{A}_k$. Left: Fixing $|\xi_0|$ (the dashed line), produces a finite number of points along the line through $x_0$ normal to $\xi_0$ where $x_0$ is mapped (with the same $\xi$). The horizontal axis is the shift relative to $x_0$ along that line. Right: The image of $(x_0,\xi_0)\in N^*\gamma$.}
\label{fig_2}
\end{figure}

\begin{remark}
The proof of Theorem~\ref{thm_art}(c) reveals something more.  The function $f_{\psi, \delta} - f_\psi$ can be regarded as the artifacts under the inversion of $\RR$. They lie outside $\mathcal{B}(0,R)$ by part (b). Moreover, they consist of the union of unitary images under $\mathcal{A}_k$ \textit{which do not intersect each other} in the following sense. Each singularity in the artifact comes from a unique one from $\WFH(f)$, and micro-localizing near $(x^\sharp,\xi^\sharp)$ allows us to recover $f$ microlocally at the unique pre-image just by applying $\sum\mathcal{A}_k$, which is non-trivial for one $k$ only. In particular, we can recover $f$ up to $O(h^\infty)$ from its artifacts outside $\mathcal{B}(0,R)$. 
\end{remark}

\begin{remark} 
As a corollary, the artifacts appear conormal to lines tangent to the edge, as in the classical case; and along each such tangent line, they stay at distance at least $2\pi/sB$ from the point of tangency. This is illustrated in Figure~\ref{fig_2}, right. In other words, $f_\delta$ is separated from the artifacts, assuming $\WFH(f)$ small enough. In a typical application of Theorem~\ref{thm_art}(c), $\mathcal{B}(0,R)$ is not going to be the computational window, it would be a much smaller neighborhood of a point $x_0$. Then the theorem applies to $f$ (micro)-localized there. On the other hand, without the localization, the reconstructed $f$ near $x_0$ could be affected by artifacts caused by singularities farther away.
\end{remark}

\subsection{Numerical examples} 

Our first example demonstrates the theorem, and in particular, the role of the magnitude $|\xi|$ of the frequency $\xi$. We take the ``coherent state''
\be{coh}
f_h (x;x_0,\xi_0)= e^{\i x\cdot \xi_0/h - |x-x_0|^2/2h}
\ee
with some $x_0$, $\xi_0\not=0$ as a test function; more precisely its real part. It is well known \cite{Zworski_book} that $\WFH(f_h)=\{(x_0,\xi_0)\}$; then  $\WFH(\Re f_h)$ also adds the point $(x_0,-\xi_0)$. In Figure~\ref{fig_coh_states_direct}, 
\begin{figure}[h!]\,\hfill
\begin{subfigure}[t]{.2\linewidth}
\centering
\includegraphics[height=.15\textheight]{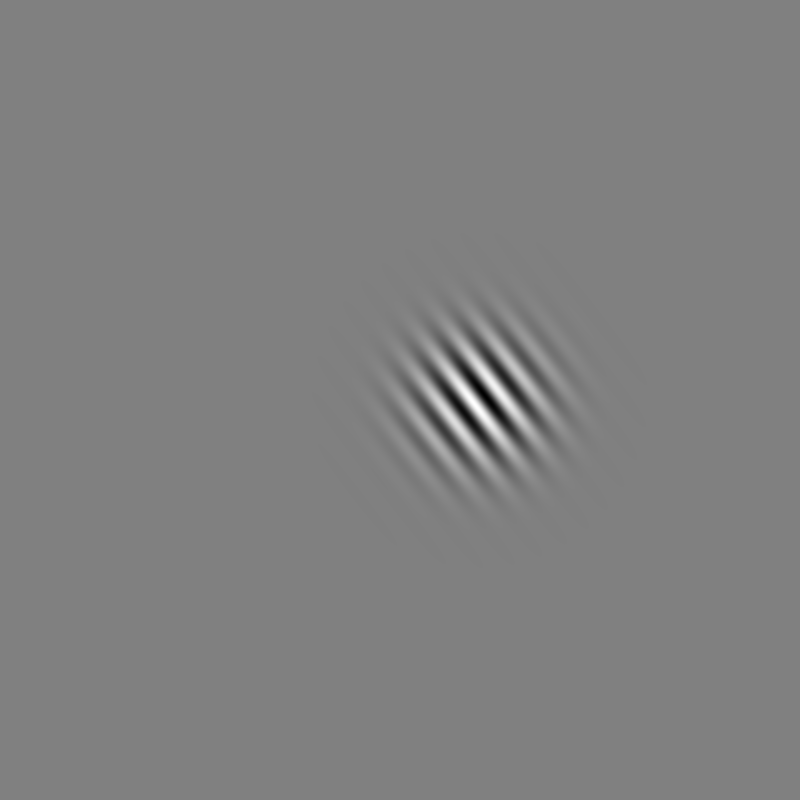}
\caption{The original $f$}
\end{subfigure}\hfill
\begin{subfigure}[t]{.2\linewidth}
\includegraphics[height=.15\textheight]{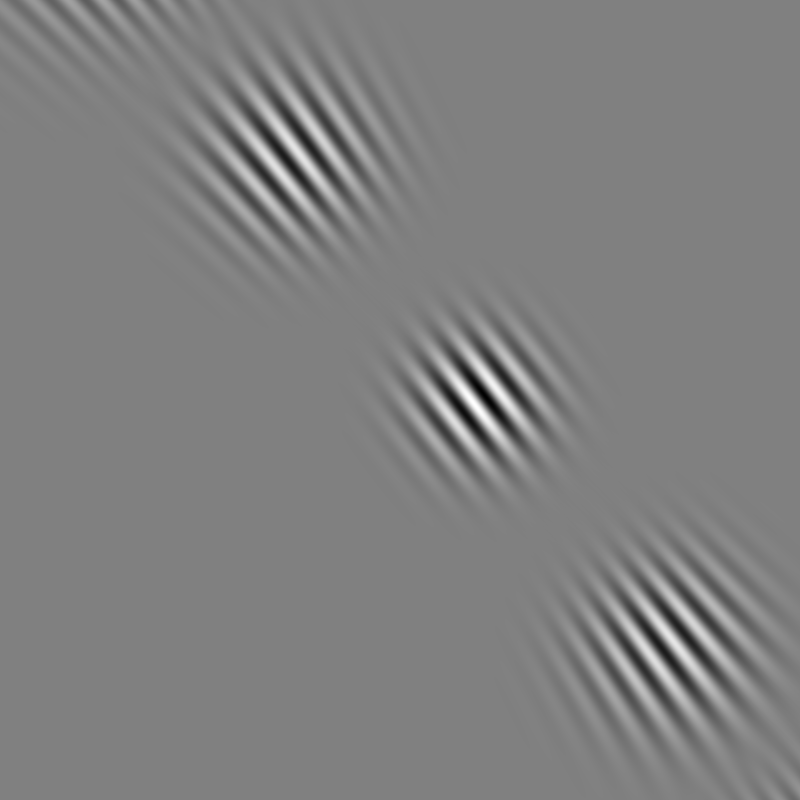}
\caption{Direct inversion with a $5^\circ$ step}
\end{subfigure}
\begin{subfigure}[t]{.2\linewidth}
\end{subfigure}\hfill
\begin{subfigure}[t]{.2\linewidth}
\includegraphics[height=.15\textheight]{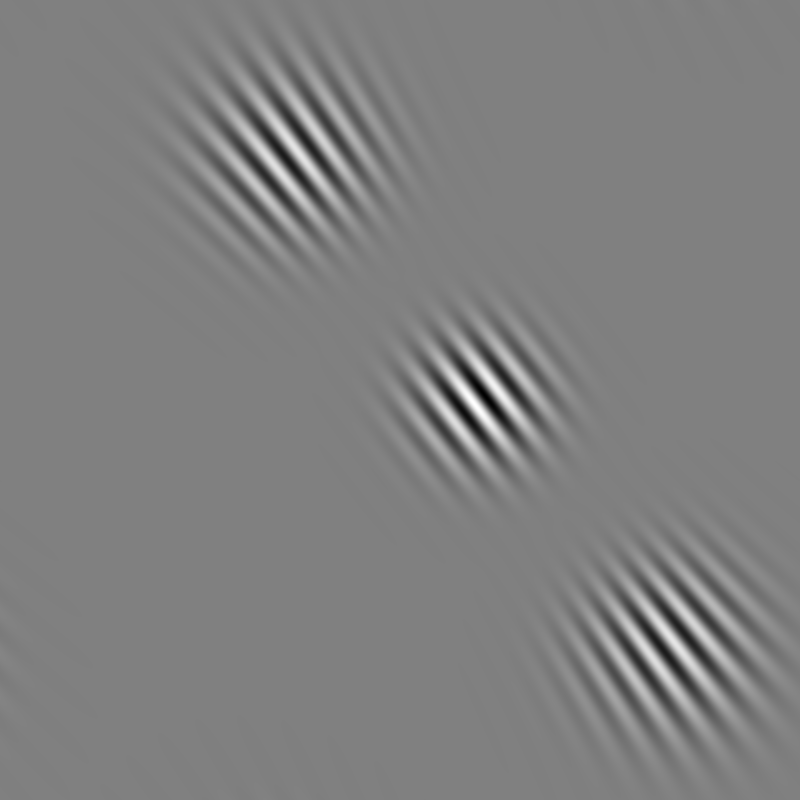}
\caption{Fourier multiplier inversion, $k=0,\pm1$.}
\end{subfigure}\hfill \,\\
\,\hfill
\begin{subfigure}[t]{.2\linewidth}
\centering
\includegraphics[height=.15\textheight]{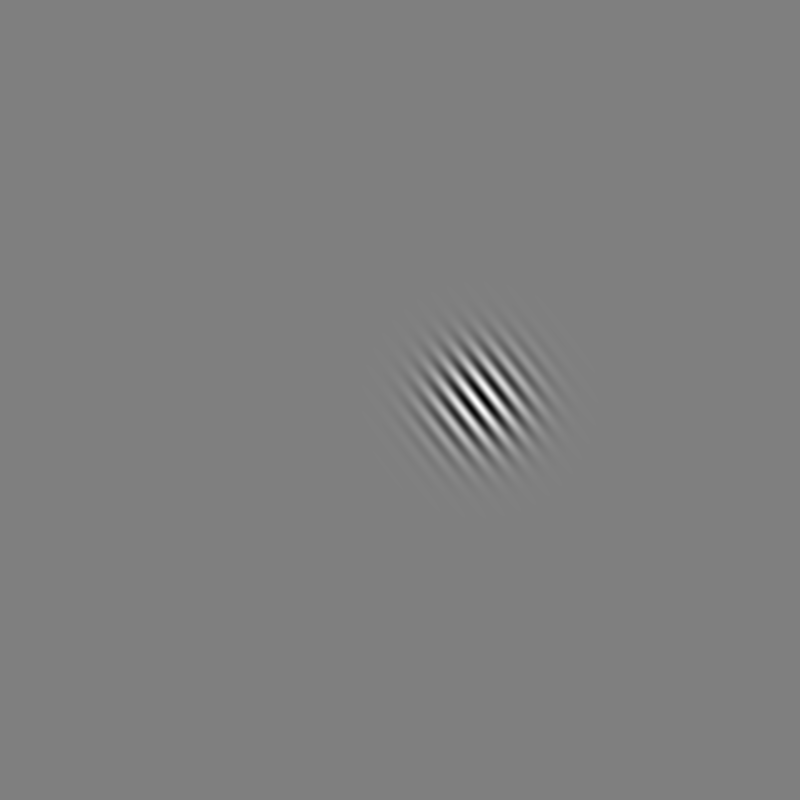}
\caption{The original $f$;  the frequency is $1.3$  higher}
\end{subfigure}\hfill
\begin{subfigure}[t]{.2\linewidth}
\includegraphics[height=.15\textheight]{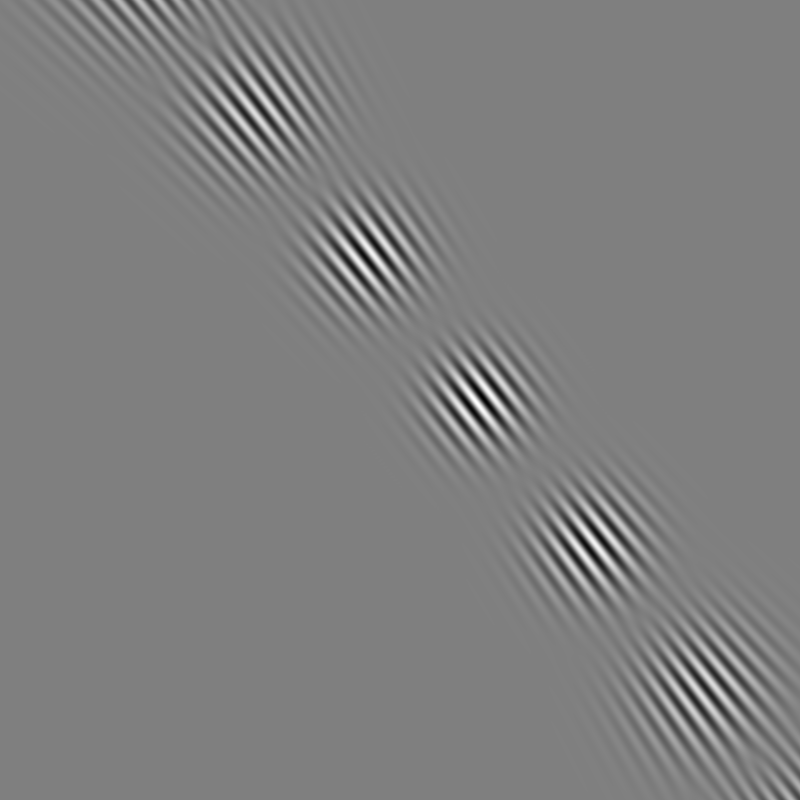}
\caption{Direct inversion with a $5^\circ$ step}
\end{subfigure}
\begin{subfigure}[t]{.2\linewidth}
\end{subfigure}\hfill 
\begin{subfigure}[t]{.2\linewidth}
\includegraphics[height=.15\textheight]{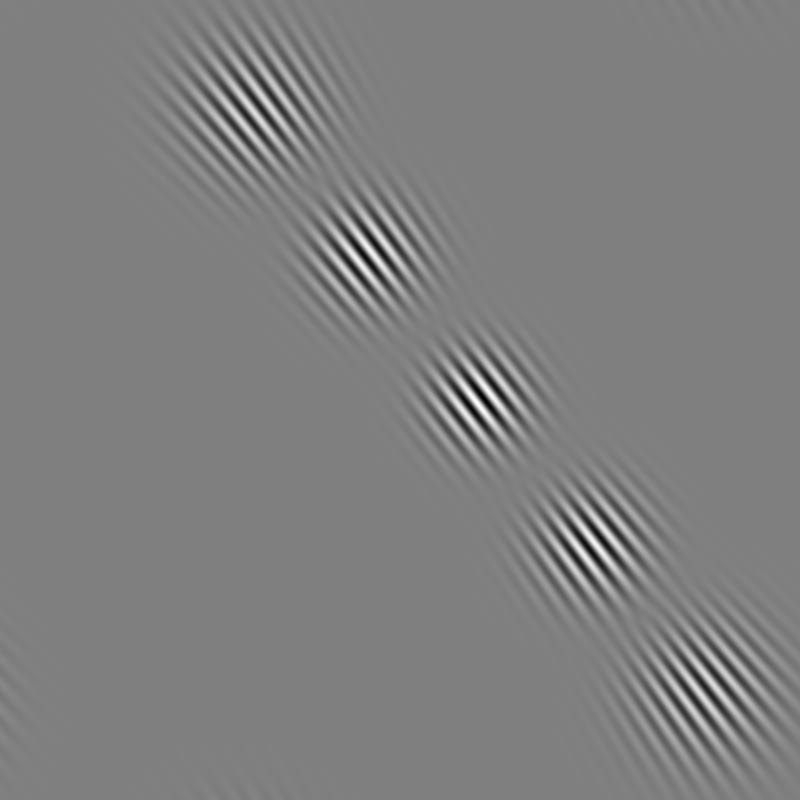}
\caption{Fourier multiplier inversion, $-2\le k\le 2$.}
\end{subfigure}\hfill\,
\caption{The direct method. A coherent state with two different frequencies $\xi$. The shifted artifacts stay at distances inversely proportional to $|\xi|$. Here, $sh = 2\pi/72$ (i.e., $5^\circ$). 
}\label{fig_coh_states_direct}
\end{figure}
we plot two such coherent states with different $|\xi|$ on the left, their computed inversions $f_{\psi,\delta}$ with $\psi=1$ in the middle, and the artifacts $\mathcal{A}_k$ on the right, with $k$ as indicated, computed as the Fourier multipliers \r{DA1}. The angular step is $5^\circ$, corresponding to $m= \pi/hs=36$. 

If we decrease $h$, we would make $f_h$ even more oscillating, we would decrease the angular step size $sh$ but the artifacts would appear at the same distances as before (and all those wave packets would be smaller because their ``width'' is proportional to $h^{1/2}$).


\section{The interpolation method, an asymptotic view}
\subsection{Asymptotic analysis of the interpolation method} 
We interpolate the data to a function of the ``continuous'' $(\varphi,p)\in [0,2\pi]\times [-R,R]$ (as usual thinking about $[0,2\pi]$ as a parameterization of the circle $S^1$, i.e., identifying $0$ and $2\pi$). There are many ways to interpolate discrete data, of course, and our choices are dictated by sampling theory requirements. Once we do that, we invert the data even though the interpolated $\RR f$ almost certainly does not belong to the range of $\RR$ by applying either the filtered backprojection  \r{FBP} to it or some other operator which is a parametrix of $\RR$. Numerically, we can just pass to a finer grid, upsample the data there, and do the inversion. 
We want to understand the resulting inversion. 


The interpolated data \r{data} where even $p$ might be discretized (but that will change soon)  then may looks like this:
\be{P1}
g_{\rm int}(\omega(\varphi),p) = \sum_{i,j} \psi(\omega( ihs_\varphi) )\RR f(\omega(ihs_\varphi),jhs_p) \chi\Big(\frac{\varphi-ihs_\varphi}{hs_\varphi}, \frac{p-jhs_p}{hs_p}\Big) ,
\ee
where the interpolation kernel $\chi$ depends on a priori assumptions on the largest \sc\ wave front set $\WFH(\RR f)$ (which in turn depends on similar assumptions on $\WFH(f)$), and on $s_\varphi$ and $s_p$, see the Appendix. One possible choice but not necessarily the only one or the best one depending on the goal, is to take $\chi$ to be product-like, i.e., to assume 
\be{P1a}
\chi(\varphi,p) = \chi_\varphi (\varphi) \chi_p(p).
\ee
 Assume that $\WFH(\RR f)$ restricts the dual variables $\hat\varphi$ and $\hat p$ to $|\hat\varphi|\le B'_\varphi <B_\varphi$ and $|\hat p|\le B_p'<B_p$, with some \sc\ band limits $B_\varphi'$, $B_p'$, where $B_\varphi$ and $B_p$ are fixed to allow for some degree of oversampling below.  A priori, they can be very close to $B_\varphi$ and $B_p$, and $B_\varphi/B_\varphi'$, and $B_p/B_p'$ can be considered as the degrees of oversampling. 
For lack of aliasing plus the  so chosen degree of oversampling, we require $s_\varphi\le \pi/B_\varphi$, $s_p\le\pi/B_p$ (the Nyquist conditions), 
the interpolation functions to be smooth, to satisfy $\supp\hat \chi_\varphi \subset [-\pi,\pi] $, and $\hat \chi_\varphi(\hat\varphi)=1$ for $|\hat \varphi|\le \pi B_\varphi'/B_\varphi$; similarly for $\hat p$. Then \r{P1} provides an approximation of $\RR f$ up to an $O(h^\infty)$ error, see the Appendix.

The critical case of no-oversampling ($B_\varphi' =B_\varphi$, $B_p'=B_p)$, which we do not allow, requires interpolation functions $\chi(s)=\sinc(\pi x)$, where $\sinc(x)=\sin x/x$. This function decays slowly and is not useful for practical implementations. On the other hand, with some oversampling, we can (and we did) chose $\chi$ to be of Schwartz class. A practical choice is the Lanczos-3 interpolation kernel 
\be{Lan3}
\text{Lan3}(x) := h_0(3-|x|)\sinc(\pi x) \sinc(\pi x/3),
\ee
where $h_0$ is the Heaviside function. While the Fourier transform of $\sinc(\pi x)$ is the characteristic function of $[-\pi,\pi]$, the Fourier transform of $\text{Lan3}$ is essentially (but not exactly, of course) supported in twice that interval but it is very close to $1$ in a half of it: in $[-\pi/2,\pi/2]$, see \cite{S-Tindel-noise}. Therefore, $\text{Lan3}$ would lead to some aliasing but it will preserve most of the non-aliased frequencies. On the other hand, if we use $\text{Lan3}(x/2)$ instead, its Fourier transform satisfies the requirements approximately with a degree of oversampling approximately $2$, while attenuating frequencies with magnitudes in $[-\pi/2, \pi]$. The resulting interpolation \r{P1} then would be $\RR f$ with a low pass filter applied, up to an $O(h^\infty)$ error. 

When we have a sampling rate $s_p$  exceeding the sampling requirements with respect to $p$, as we assume in this paper,  we just assume that we are given $\psi( ihs_\varphi ) \RR f(\omega(ihs_\phi) ,p)$ for all $p\in [-R,R]$. 
Then \r{P1} reduces to 
\be{P2}
g_{\rm int} = \sum_{j=1}^{\pi/sh} \psi( \omega(ihs)  )\RR f(\omega(jhs ) ,p) \chi\Big(\frac{\varphi-jhs}{hs }\Big),
\ee
where $s=s_\varphi$, which is the discrete data, interpolated. 
Note that $\mathcal{H}$  could be applied before or after the interpolation with the same result. 
The inversion in this case would be 
\be{I1}
f_{\psi,\chi}:= \RR'\mathcal{H} g_{\rm int} .
\ee

\begin{figure}[h!] 
\includegraphics[height=.16\textheight, trim={0 25 0 20},clip]{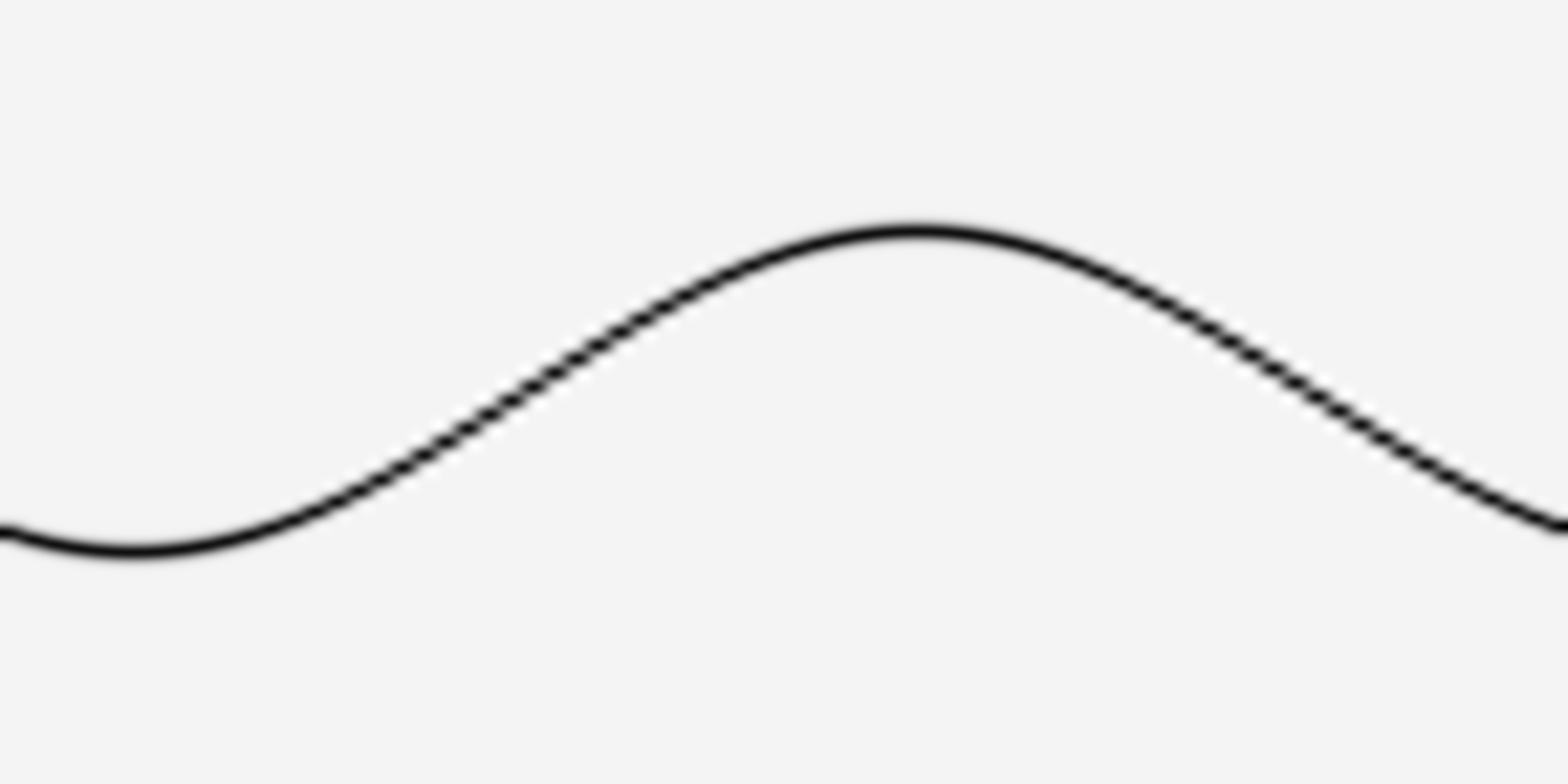}
\caption{The sinogram of a very small Gaussian, sampled at $5^\circ$ in $\varphi$, then interpolated. In the steep parts of the curve, the interpolation is not optimal but it is good along the horizontal ones.}\label{fig_sin}
\end{figure}

One of the results in \cite{S-Sampling} is that if $\supp f\subset \mathcal{B}(0,R)$, the sharp condition for avoiding aliasing is 
\be{F1}
s  < \pi/(RB),
\ee
where $B$ is the band limit for $f_h$. In that case, one can reconstruct $\mathcal{R}_\kappa f$ up to $O(h^\infty)$, and then $f$, using $\chi$ as described after \r{P1a} 
in the $\varphi$ variable. 
If $s <\pi/(2RB)$ (oversampling by a factor of two), one can use the Lanczos-3 interpolation which is local and much more practical, to achieve great accuracy. When \r{F1} does not hold, and one still uses the same reconstruction, aliasing occurs. The aliasing artifacts appear as 
a sum of   h-FIOs with canonical relations which happen to be the same as \r{CR}, 
when \be{RA2}
x\cdot\xi^\perp+ 2k\pi/s \in [-\pi/s , \pi/s ], 
\ee
(and we  relabeled them by changing the sign of $k$ compared to \cite{S-Sampling}). 
If \r{F1} holds, we have $|x\cdot\xi^\perp|<RB<\pi/s$, thus \r{CR} can hold with $k=0$ only, hence no aliasing. 

By \r{RA2}, $k$ depends on $(x,\xi)$. Note that for each $(x,\xi)$ with $\xi\not=0$, there is unique $k$ satisfying \r{RA2} with the exception of the case when the left-hand side happens to be an endpoint on the interval on the right; but then $\hat\chi$ kills the interpolation for such frequencies because we assume $\supp \hat\chi_\varphi\subset [-\pi,\pi]$. On the other hand, if we use the Lanczos-3 kernel, which does not satisfy this condition, even approximately (but it does in $[-2\pi,2\pi]$, as explained above), we can  get two aliased artifacts.

In Figure~\ref{fig_1}, we illustrate this analysis. We have $x\cdot\xi^\perp>0$, so if aliasing happens, we must have $k<0$ in \r{RA2}. Then that $(x,\xi)$ would  shift along the ray issued from $x$ tangent to the curve $\gamma$, in the direction $-\xi^\perp$, i.e., towards the point on that tangent line closest to the origin. The jump at $\gamma$ may creates a singularity at $(x,-\xi)$ as well but then $k>0$ and the artifact would still appear on the same ray.

\begin{figure}[h!] 
  \centering
	\includegraphics[page=1]{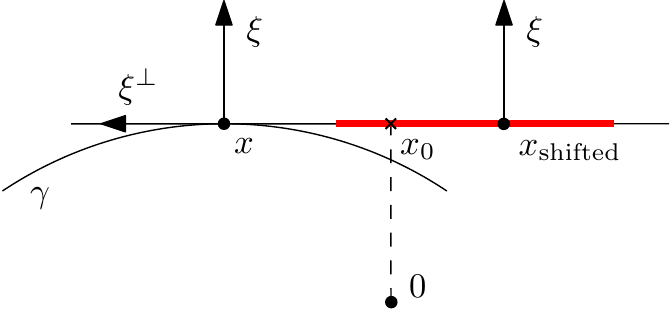}
\caption{\small Aliasing with the interpolaton method: $(x,\xi)$ shifts to $(x_{\rm shifted},\xi)$ along the line tangent to $\gamma$ at $x$. The point $x_{\rm shifted}$ can be at distance at most $d/3$ to the left and $d$ to the right of $x_0$, where $d=|x-x_0|$.}
\label{fig_1}
\end{figure}


A closer inspection of conditions \r{CR}, \r{RA2} reveals that the aliased $(x,\xi)$ may appear at points $x$ over some interval over that tangent line only. 
Indeed, condition \r{RA2}  is equivalent to 
\be{F3'}
- x\cdot\xi^\perp \in [(2k-1)\pi/s , (2k+1)\pi/s ].
\ee
 Then $x$ is displaced along the  line through $x$ in the direction of $\xi^\perp$ by $2\pi k /(s |\xi|)$, and for the shifted $x$ we have \[
-\Big(x + \frac{2\pi k}{s } \frac{\xi^\perp}{|\xi|^2} \Big)\cdot \frac{\xi^\perp}{|\xi|} \in \frac1{|\xi|}[-\pi/s, \pi/s].
\] 
Write $x\cdot\xi^\perp = x\cdot(\xi^\perp/|\xi|)|\xi|$ in \r{F3'}, and  determine the range of $|\xi|$ with $\xi/|\xi|$ fixed,  assuming first $-x\cdot\xi^\perp>0$, hence $k>0$. We see then that  $x$ can shift along the line through $x$ in the direction $\xi^\perp/|\xi|$  within the range of the signed distances 
\be{F3}
-\frac{x\cdot\xi^\perp}{|\xi|}\Big[\frac{2k}{2k+1} , \frac{2k}{2k-1} \Big],\quad k=1, 2,\dots,\quad x\cdot\xi^\perp>0,
\ee
where we used the convention $k(a,b)=(ka,kb)$. When $-x\cdot\xi^\perp<0$ (then $k<0$),  we get the same conclusion just by replacing $\xi$ by $-\xi$, and $k$ by $-k$.  Thus \r{F3} holds for $k<0$ , as we indicated above.  
In Figure~\ref{fig_1}, we have $k=-1$. 

The largest of the intervals in the square brackets in \r{F3} is $[2/3,2]$. The upper bound of the distance $2 |x\cdot\xi^\perp|/|\xi|$ is achieved right when aliasing occurs, i.e., when $k=1$ in \r{RA2} and the l.h.s.\ approaches $-\pi/s$. If we keep the direction of $\xi$ the same but increase its magnitude, the aliased singularity moves closer to $x$ until it gets at distance $(2/3) |x\cdot\xi^\perp|/|\xi|$. Then it jumps to $4/3$ of that factor, moves to $4/5$, etc. In the end, the minimal interval is $[2/3,2]$. 
Therefore, we have the following.

\begin{theorem}\label{thm_2}
Under the conditions of Theorem~\ref{thm_art}, statements (a) and (b) there are preserved for $f_{\psi,\chi}$. Instead of an equality in (c), we have  
\be{art1a}
\begin{split}
\WFH (f_{\psi, \delta}  )\setminus 0 &\subset \bigg\{x- \frac{x\cdot\xi^\perp}{|\xi|} \bigg[\frac23, 2\bigg]\frac{\xi^\perp}{|\xi|}, \; (x,\xi)\in \WFH(f_\psi)\setminus 0 \bigg\}. 
\end{split}
\ee
\end{theorem}

\subsection{Translation non-invariance and refocusing} 
One of the consequences of the ana\-lysis in \cite{S-Sampling} is that the resolution, defined there, is inversely proportional to $|x|$ (and also direction dependent). This is also consistent with \r{F1}, where $R\ll1$ allows $B\gg1$ for the same step $s$. This makes the origin a special point, with the resolution near it the highest. In a way, the interpolation method is ``focused'' at the origin.  
It is easy to see that the parallel geometry parameterization is not invariant under translations and rotations in the sense that it does not preserve its form. Rotations $x\mapsto Ux$ are innocent; they just transform $\omega$ into $\tilde\omega:= U^*\omega$. In \r{omega}, this corresponds to shifting $\varphi$ (and still considering in modulo $2\pi$). Shifting $x$ by $\tilde x= x-x_0$ however, changes the type of the equation $x\cdot \omega(\varphi) =p$ to 
\be{T2}
\tilde x\cdot \omega(\varphi) =p -x_0\cdot \omega(\varphi) .
\ee
Setting 
\be{Tp}
\tilde p = p -x_0\cdot \omega(\varphi),
\ee
we see that $\tilde p$ depends on $\varphi$ now. This is reasonable to expect: each time we choose an angle $\varphi$, we are free to put the origin on the line $\omega(\varphi)^\perp$ parameterizing the lines with that direction, anywhere we want to.  In \r{T1}, the choice happens to correspond to the line through the origin in the $x$-plane. This makes the origin a special point without any need to be such. We are free to change that parameterization to \r{T2}, for example, to even do something different, choosing $\tilde p$ to be a more general function of $\varphi$. 

That freedom does not do much when we have $\mathcal{R}f(\omega,p)$ for all $\omega$ and $p$ (or for them in some open set). In the discrete setting however, things change. We will call the re-parameterization \r{Tp} \textit{refocusing}. If we know $\RR f(\omega,p)$ for $\omega$ in a discrete set (and all $p$), we can perform \r{Tp} for each such $\omega$, and $x_0$ fixed. This would map the curve $x_0 \cdot\omega=p$, see Figure~\ref{fig_sin} into the straight line $\tilde p=0$. 
Then the inversion would look like $x_0$ were the origin, which would move the aliasing artifacts elsewhere! Recall that we assume a high enough sampling rate $s_p$, which makes implementing \r{Tp} easy.

\subsection{Relation between the two methods}
Finally, we show that the interpolation reconstruction operator is just the ``direct'' one convolved in the $\omega$ variable with the interpolating function. 
\begin{theorem}  \label{thm_conv}
For every $f\in C_0^\infty(\R^2)$, 
\[
f_{\psi,\chi}(r\theta)  = \chi_{sh}  *_{\theta} f_{\psi,\delta}(r\theta),
\]
where $*_\theta$ is the circular convolution in the $\theta$ variable, and $\chi_h(\theta)= h^{-1}\chi(\theta/h) $.  
\end{theorem}

\begin{proof}
By \r{P2}, \r{I1},
\[
f_{\psi,\chi}(x) = \int \sum_{j} \psi( jsh )(\mathcal{H}\RR f)(\omega(jsh) ,x\cdot\omega(\varphi ) )\chi\Big(\frac{\varphi-jsh }{sh }\Big)\d\varphi.
\]
Write $x=r\omega(\theta)$, and make the change of variables $\tilde \varphi = \varphi-jsh $. Since $x\cdot\omega(\tilde\varphi +jsh) = r \omega(\theta)\cdot \omega(\tilde\varphi +jsh) = r \omega(\theta - \tilde\varphi )\cdot \omega(jsh)$, we get
\[
f_{\psi,\chi}(r\omega(\theta)) = sh\int \sum_{j} \psi( jsh )(\mathcal{H}\RR f) (\omega(jsh) , r \omega(\theta - \tilde\varphi )\cdot \omega(jsh)  )\chi_{sh}(\tilde\varphi )\, \d\tilde \varphi.
\]
This is exactly the circular convolution of \r{f_delta_d'} with $\chi_{sh}$ as claimed. 
\end{proof}

The convolution in Theorem~\ref{thm_conv} is a Fourier multiplier in polar coordinates, in the radial variable,  with $\mathcal{F}_h\chi_s$. Passing back to the Cartesian coordinates, we get an \HPDO\ with principal symbol $\chi_s(-x^2\xi_1 + x^1\xi_2)$ at least away from $x=0$. This shows that the two reconstructions are related by an \HPDO, and since we showed in Theorem~\ref{thm_1} that $f_\psi\mapsto f_{\psi,\delta}$ is an $h$-FIO, it now follows that $f_\psi\mapsto f_{\psi,\chi}$ is an $h$-FIO with the same canonical relation, something we proved directly in Theorem~\ref{thm_2}.

\subsection{Comparison of the two methods} 

We managed to get from the discrete measurements $g_j$ \r{data} to ``continuous'' ones with the aid of the Poisson summation formula \r{S4}. We will offer here an alternative point of view. 

We can think of the numerical integration formula \r{f_delta_d} in the following way. First, we interpolate the discrete data somehow with an interpolation kernel $\chi$ having total integral one. We do that for each $x$, along the curve $\varphi\mapsto (\omega(\varphi), x\cdot\omega(\varphi))$, see Figure~\ref{fig-SL}. Then integrating the interpolated function removes $\chi$ and reduces to the finite sum  \r{f_delta_d}. 

The interpolation method, on the other hand, interpolates horizontally in  Figure~\ref{fig-SL}, i.e., along the lines $\varphi\mapsto (\omega(\varphi), p=\text{const.})$. Only one of those lines coincides with some of the lines above: the line $p= 0$ (in the $(\varphi,p)$ plane, more precisely, on the cylinder $S^1\times\R$), which corresponds to all lines in the $x$ plane through $x=0$. The two methods are equivalent, roughly speaking, in an infinitesimal neighborhood of $x=0$, as Theorem~\ref{thm_conv} indicates as well. Away from $x=0$, $f_{\psi,\chi}$ is just an angularly blurred version of $f_{\psi,\delta}$.  The advantage of the direct method is that the interpolation before integration (which is not needed, as explained above) is $x$-dependent. In that sense, that methods focuses at every point $x$ to evaluate $f_{\psi,\delta}$ there.

\section{Recovery of an edge and aliasing from an edge, classical view}  \label{sec-classical}
Assume that $f$ is piecewise smooth with a jump over a smooth curve (an ``edge'') near some point $x_0$, and has no other singularities. We want to understand how well the edge is resolved, and what kind of aliasing artifacts are created. We want to emphasize that if $f$ has other singularities, they may create aliasing artifacts near $x_0$ as well, interfering with the ones we analyze here. 

Assume that we use the direct method, formula \r{f_delta}, which is also \r{f_delta_d} when we restrict our attention to lines close to being tangent to the edge, and $\psi=1$ there. It is enough to analyze each summand in \r{f_delta} independently. We are going to analyze three cases which do not exhaust all possible ones. In this section, $h=1$, i.e., we do not take the angular step to be a small parameter, respectively $m$ is fixed. We study the direct method here only. 

\subsection{A flat edge} Assume that the edge is flat neat $x_0$. Then the recovered $f_\delta$ depends on whether that edge is normal to some of the $\omega_j$'s (i.e., parallel to some of the lines in our family) or not; and in the latter case, it will depend to the distance of its normal to $\{\omega_j\}$.   

Assume first that the edge is normal to $\omega_{j_0}$ for some $j_0$. Then $\RR f(\omega_{j_0},p)$ would have a jump-type of singularity at some $p=p_0$, and $\RR f(-\omega_{j_0},p)$ would have a jump-type of singularity at  $p=-p_0$. 
Without loss of generality, we can assume that it is the former term appearing in \r{f_delta}. 
Then $\mathcal{H} \RR f(\omega_{j_0},p)$ would be a distribution but not a (locally $L^1$)  function! Indeed, writing $\RR f(\omega_{j_0},p)=kh_0(p-p_0)$, $k\not=0$, modulo higher regularity terms (which regularity depends on the behavior of $f$ near that edge), where $h_0$ is the Heaviside function, one needs to understand $\mathcal{H} \RR f(\omega_{j_0},p)\sim k\mathcal{H} h_0(p-p_0)$. Therefore, the leading singularity of $f_\delta$ would be expected to be (ignoring the localization for a moment, 
\[
\begin{split}
\frac{2\pi}m k\mathcal{H} h_0(p-p_0) &= \frac{k}{2m} H\d_p h_0(\cdot-p_0)= \frac{k}{2m} H\delta(\cdot-p_0) \\
& =  \frac{k}{2\pi m}\,  \pv  \frac{1}{p-p_0}, \quad \text{where $p=x\cdot\omega$}.
\end{split}
\]
In \r{L1}, we provide a more precise statement, and a second term. 
Note that this is the behavior along the line $x\cdot\omega_{j_0}=p_0$ independently of whether the point on that line is on the actual edge or not. The result is a distribution. All other terms in \r{f_delta}  would contribute smooth terms, so this describes all leading order singularities of $f_\delta$ under our assumptions.

If the edge is not normal to any $\omega_{j}$, then $f_\delta$ would be smooth. When the edge is ``almost normal'' to some $\omega_{j}$ however, there will be steep change across that line.  

In Figure~\ref{fig_flat-edge}, we demonstrate this behavior. The computations are done in a $1000\times 1000$ grid. 
The angular step is $10^\circ$, with the vertical direction being among the set of the directions (corresponding to $\varphi_1=0$ being the first one).  
The reason the bright phantom looks so pale in (b) is that the range has been adjusted from $[0,1]$ in (a) to $[-2.1, 2.7]$.  Two cross-sections are plotted. The one through the maximum of the jump actually recovers the edge well, plus a $\pv (1/x)$ type of singularity as predicted. 
\begin{figure}[h!]\hfill 
\begin{subfigure}[t]{.2\linewidth}
\centering
\includegraphics[height=.17\textheight]{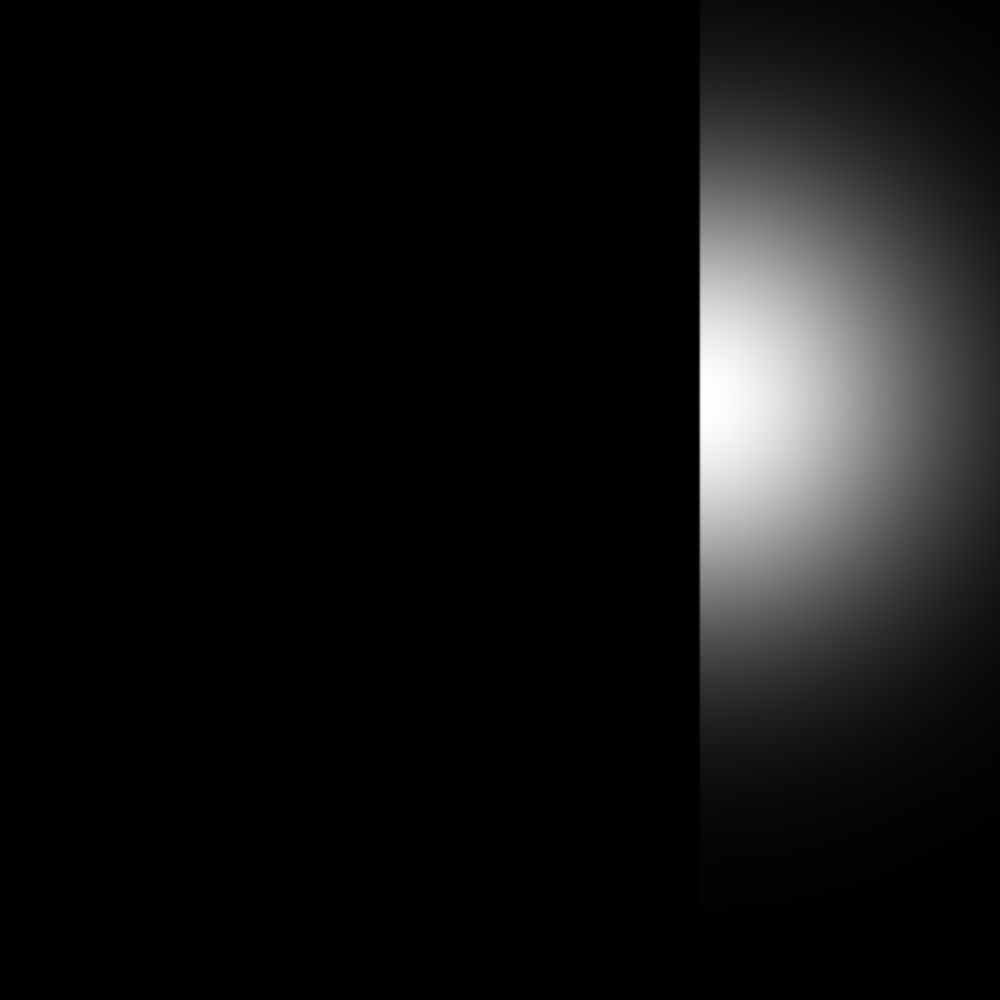}
\caption{A flat edge}
\end{subfigure}
\hfill
\hfill
\begin{subfigure}[t]{.2\linewidth}
\centering
\includegraphics[height=.17\textheight]{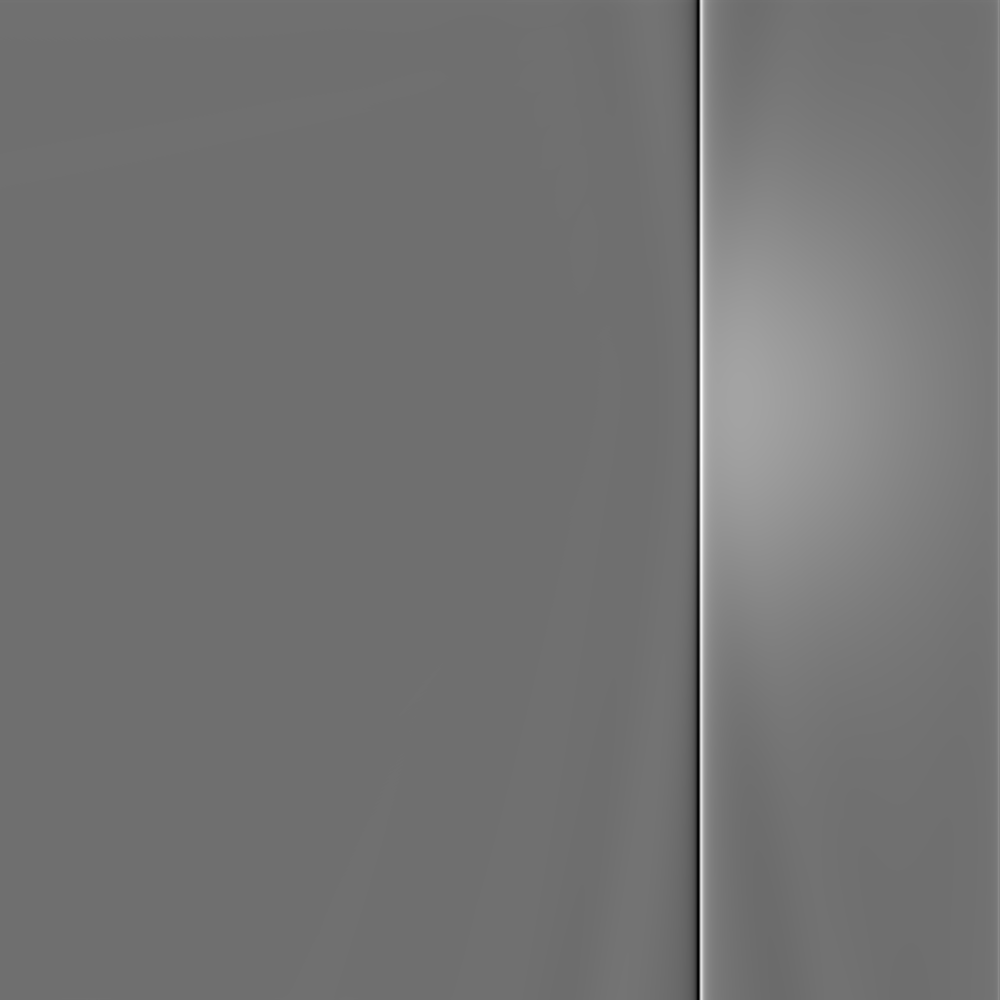}
\caption{Recovered, angular step  $10^\circ$.}
\end{subfigure}\hfill \hspace{0.2in}\,{}\\  
\hfill 
\begin{subfigure}[t]{.38\linewidth}
\centering
\includegraphics[height=.15\textheight]{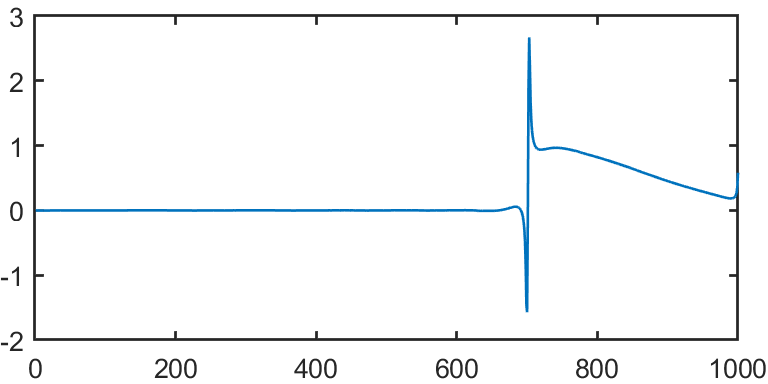}
\caption{A horizontal cross-section plot near the maximal jump.}
\end{subfigure}\ \ \ \ \hfill  \ \ \ \ 
\begin{subfigure}[t]{.38\linewidth}
\centering
\includegraphics[height=.15\textheight]{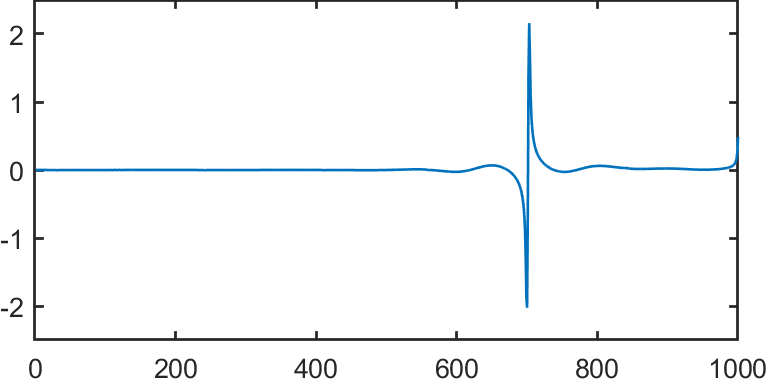}
\caption{A horizontal cross-section plot near the bottom.}
\end{subfigure}\hfill\hfill{}
\caption{A flat edge with $s=10^\circ$ with $\pv (1/x)$ singularities visible in the reconstruction. Near the actual edge, the jump is recovered well but this is due to smooth but sharply changing contributions from close non-tangent lines.
}\label{fig_flat-edge}
\end{figure}
The edge is well recovered because the contributions from the lines with directions close to vertical are smooth but sharply changing near the edge. This is better understood in asymptotic sense, when the angular step size gets smaller and smaller, as we do later. The second cross-section is near the bottom of the square.  

\subsection{A strictly convex/concave edge} 
Assume that the edge is strictly convex or concave, depending on the direction at which we are looking, i.e., it is a smooth curve with nonzero curvature near $x_0$. Then $\RR f$ would have singularities at lines tangent to the edge, where $f$ jumps. Fixing one such direction, $\RR f(\omega,p)\sim k h_0  (p-p_0)^{1/2}_+$, where $t_+=\max(t,0)$, $k=2\sqrt2f(x_0)/\sqrt{|\kappa|}$, and $\kappa$ is the curvature. Again, without loss of generality we assumed that the curve lies in $p\ge p_0$, not $p\le p_0$.  
As above, we need to understand $\mathcal{H}$ applied to it.  
This is done in \r{L2} in Lemma~\ref{lemma_H}. 
We get that $f_\delta$ would have conormal singularities along the line determined by $(\omega_{j_0},p_0)$ of the kind
\[
-\frac{k}{2\pi m}(p-p_0)_-^{-1/2}
\]
as the most singular part of $f_\delta$, near the line $x\cdot\omega_j=p_0$. This is an integrable singularity. A numerical reconstruction is shown in Figure~\ref{fig_convex-edge}.
\begin{figure}[h!]\hfill 
\begin{subfigure}[t]{.2\linewidth}
\centering
\includegraphics[height=.17\textheight]{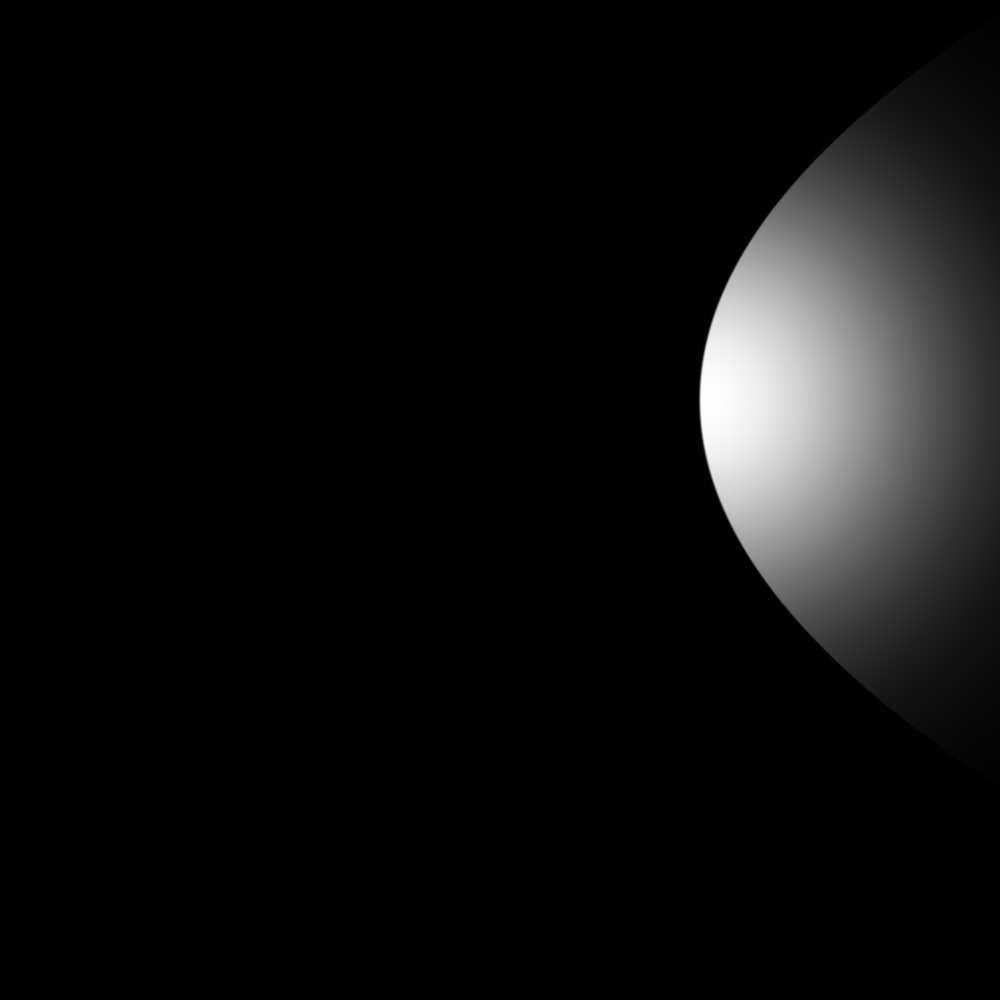}
\caption{A convex edge}
\end{subfigure}
\hfill
\hfill
\begin{subfigure}[t]{.2\linewidth}
\centering
\includegraphics[height=.17\textheight]{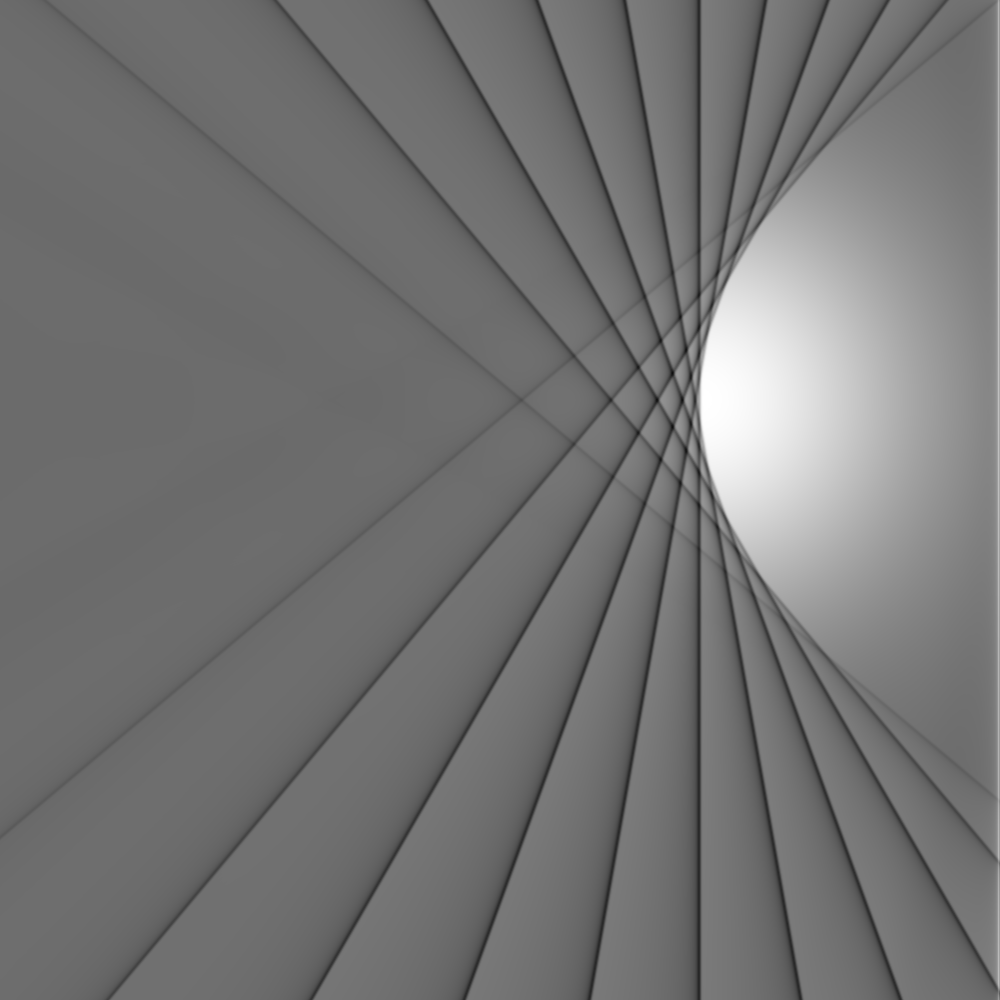}
\caption{Recovered, angular step  $10^\circ$.}
\end{subfigure}\hfill \hspace{0.2in}\,{}\\  
\hfill 
\begin{subfigure}[t]{.38\linewidth}
\centering
\includegraphics[height=.15\textheight]{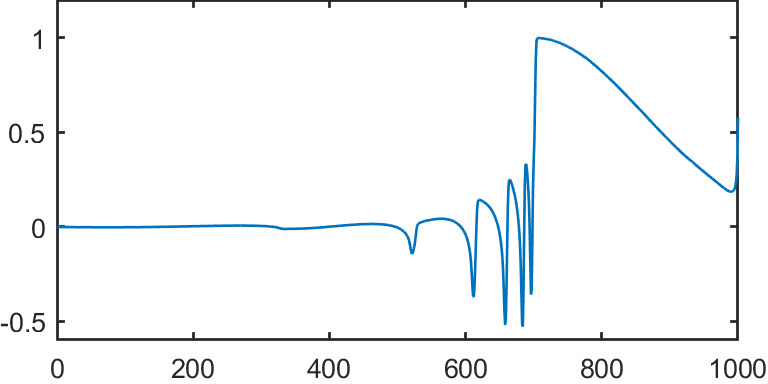}
\caption{A horizontal cross-section plot near the maximum of the jump.}
\end{subfigure}\ \ \ \ \hfill  \ \ \ \ 
\begin{subfigure}[t]{.38\linewidth}
\centering
\includegraphics[height=.15\textheight]{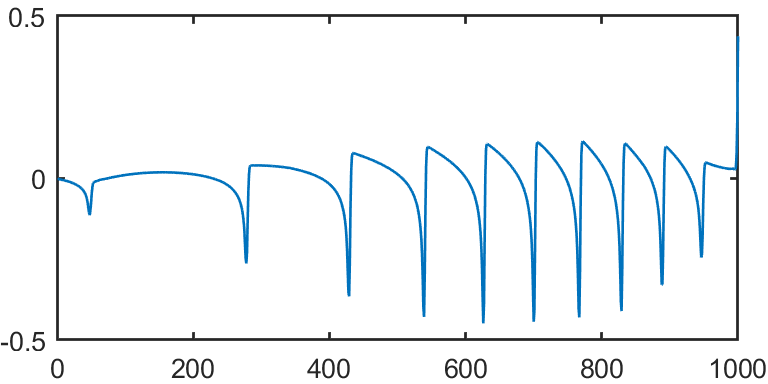}
\caption{A horizontal cross-section plot near the bottom.}
\end{subfigure}\hfill\hfill{}
\caption{A convex edge with $s=10^\circ$ with $-x_-^{-1/2}$ singularities visible in the reconstruction. Near the actual edge, the jump is recovered well but this is due to smooth but sharply changing contributions from close non-tangent lines.
}\label{fig_convex-edge}
\end{figure}
The  $-x_-^{-1/2}$ singularities are well visible.

\subsection{Artifacts from a corner} 
Let $f$ has a jump across a corner, like $f=h_0(x^1)h_0(x^2)$ near $x=0$. Then $\WF(f)$ over the corner consists of all directions, which will create singularities conormal to all lines in our set through this corner. To be more precise, assume that we have two smooth curves through $x_0$, intersecting transversally, so that $f$ is equal to the restriction of a smooth function $f_0$ with $f_0(x_0)\not=0$, to one of the four sectors, and zero in the other three. Assume that $\omega_{j_0}$ is not normal to either of those curves at $x_0$. Then $\RR(\omega_{j_0},p)\sim k  (p-p_0)_+$ locally, $k\not=0$, modulo smoother terms.   By \r{L3} in Lemma~\ref{lemma_H}, $f_\delta$ would have conormal singularities along the line determined by $(\omega_{j_0},p_0)$ of the kind
\[
\frac{k}{2\pi m} \log|p-p_0|.
\]
It is the weakest of the three. 

A numerical illustration is presented in Figure~\ref{fig_corner-edge}. In (d), we  see log type of peaks along a horizontal line staying at 30\% from the bottom. 
\begin{figure}[h!]\hfill 
\begin{subfigure}[t]{.2\linewidth}
\centering
\includegraphics[height=.17\textheight]{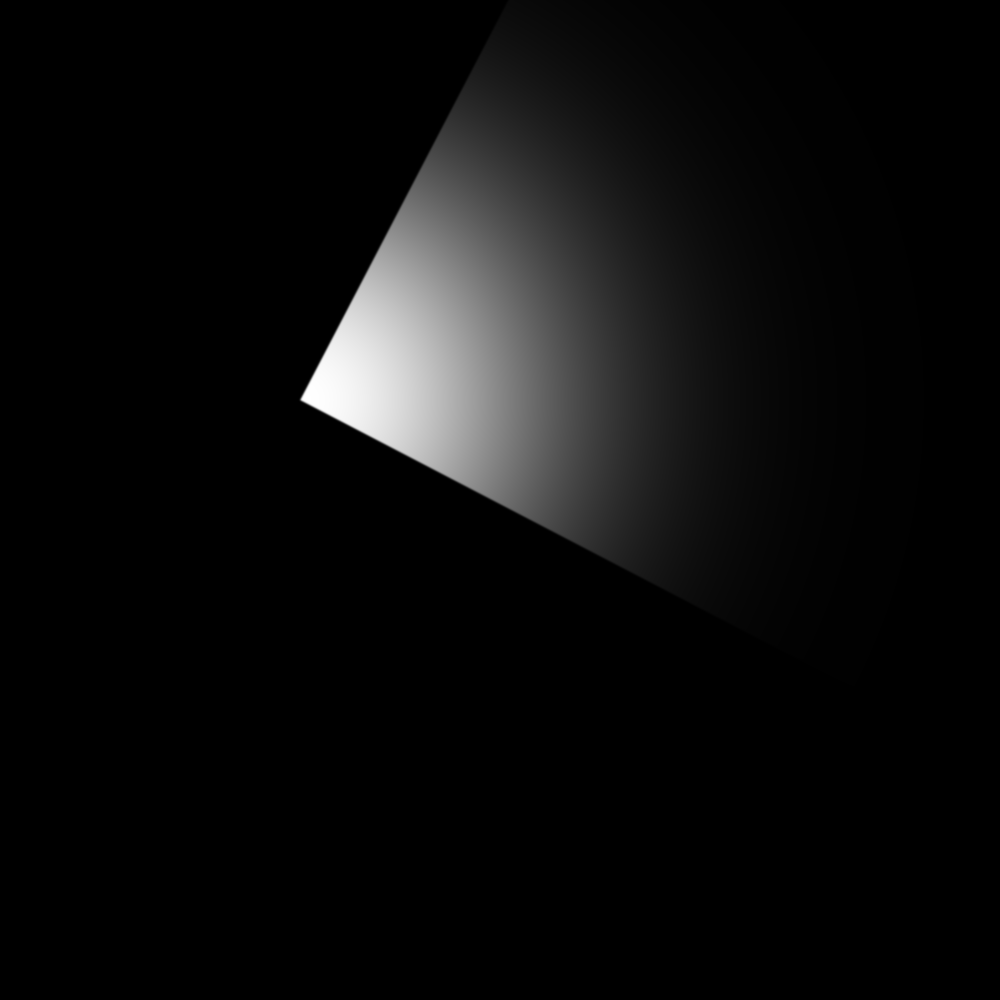}
\caption{A corner}
\end{subfigure}
\hfill
\hfill
\begin{subfigure}[t]{.2\linewidth}
\centering
\includegraphics[height=.17\textheight]{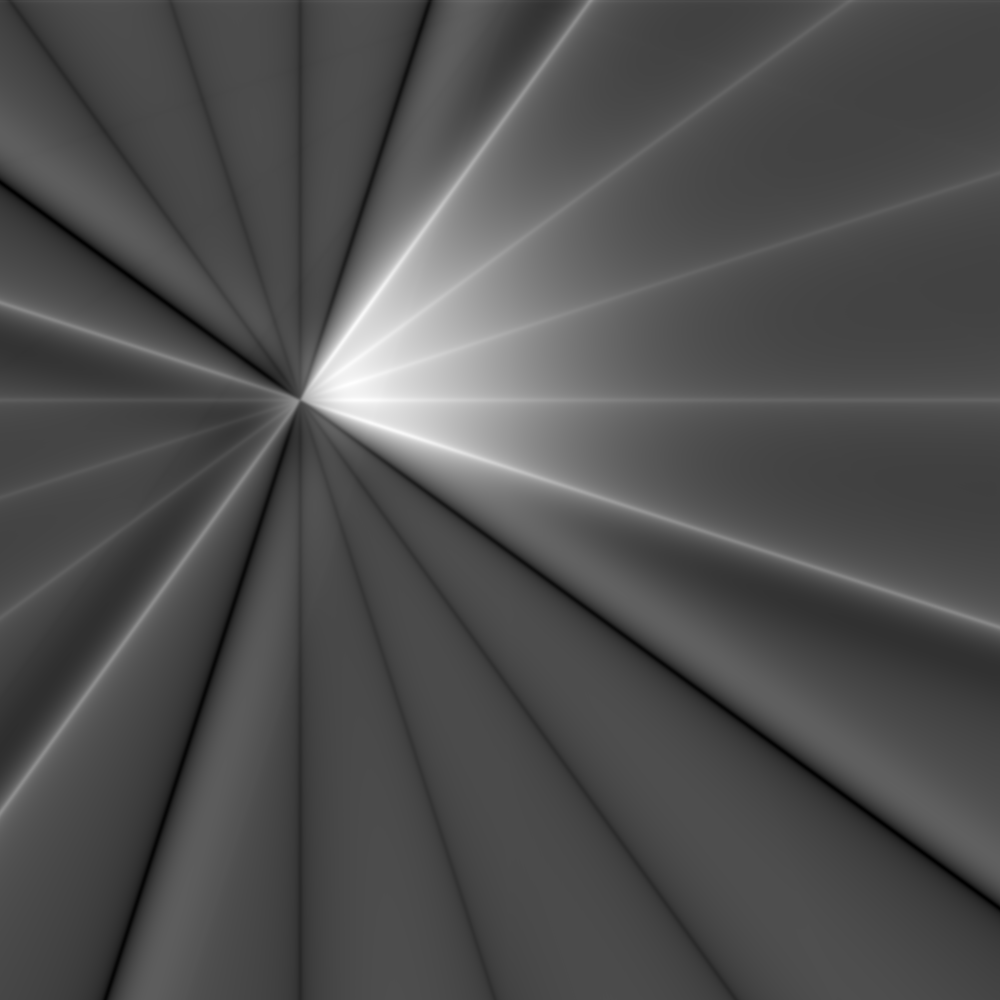}
\caption{Recovered, angular step  $18^\circ$.}
\end{subfigure}\hfill \hspace{0.2in}\,{}\\  
\hfill 
\begin{subfigure}[t]{.38\linewidth}
\centering
\includegraphics[height=.15\textheight]{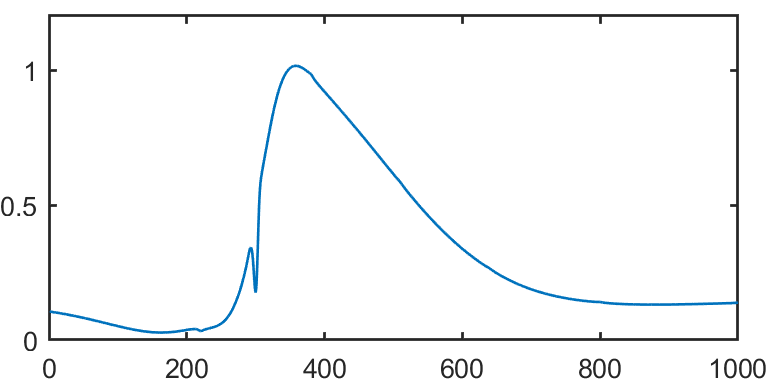}
\caption{A horizontal cross-section plot near the tip.}
\end{subfigure}\ \ \ \ \hfill  \ \ \ \ 
\begin{subfigure}[t]{.38\linewidth}
\centering
\includegraphics[height=.15\textheight]{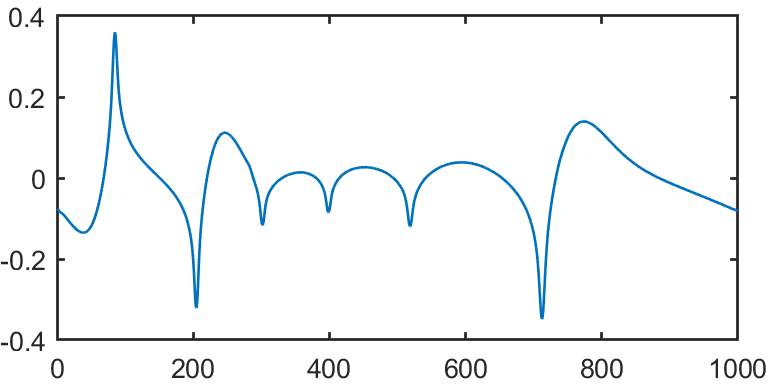}
\caption{A horizontal cross-section plot near the bottom.}
\end{subfigure}\hfill\hfill{}
\caption{A corner with $s=18^\circ$ with $\pm \log|x|$ singularities visible in the reconstruction. Near the actual edge, there is a similar singularity as well plus a blurred version of the  jump, which is one.
}\label{fig_corner-edge}  
\end{figure}
 Most of them point down, corresponding to $+\log|\cdot|$. They correspond to lines through the corner not entering the sector where $f>0$. The most left one corresponds to a line through the corner entering that sector, and the singularity is of the type $-\log|\cdot|$. This explains why that peak points upwards.

We used the following lemma above. 
\begin{lemma}\label{lemma_H}
Let $\phi\in C_0^\infty(\R)$. Then 
\begin{align}
\label{L1}
 \mathcal{H}\phi h_0 (x)&= \frac1{4\pi^2}  \left(\phi(0)\pv \frac{1}{x} +  \phi'(0) \, \log|x|\right) \mod C^0(\R),\\  \label{L2}
  \mathcal{H}\phi x_+ &= \frac1{4\pi^2}   \phi(0) \log|x| \mod C^0(\R),\\  \label{L3}
 \mathcal{H}\phi x_+^\frac12  &= -\frac1{4\pi^2}   \phi(0) x_-^{-\frac12} \mod C^0(\R). 
\end{align}
\end{lemma}
\begin{proof}
The lemma is a computation of a singularity conormal at $x=0$ under the action of the \PDO\ $\mathcal{H}= (4\pi)^{-1}|D|$. The result is given by \cite[Theorem~18.2.12]{Hormander3}. In our case, a (compactly supported) conormal distribution in $\R$ at $x=0$ of order $m$ is given by
\[
u(x) = \frac1{2\pi}\int e^{\i x\xi} a(\xi)\,\d\xi,
\]
i.e., just the inverse Fourier transform of $a$, where $a$ is a symbols of order $m+1/4$. Applying various smooth cutoffs to $u$ which are equal to one near $x=0$ may change the symbol only up to a term of order $-\infty$. If $\phi$ does not satisfy that, the symbol would be modified depending on the Taylor expansion of $\phi$ at zero. 
The distributions in the lemma are not compactly supported before multiplying by $\phi$ but they are homogeneous, thus they have homogeneous Fourier transforms singular at $\xi=0$ only. One can see that a multiplication by $\phi(x)$ would produce a compactly supported conormal distribution with a symbol equal to those Fourier transforms, up to $O(|\xi|^{-\infty})$. 

Applying a \PDO\ $p(x,D)$ of order $m'$ to a conormal distribution of order $m$ in our case produces a conormal distribution at $x=0$,  again of order $m+m'$, with complete symbol
\be{L3a}
\sum \langle -\i D_x,D_\xi\rangle ^j p(x,\xi)a(\xi)/j!|_{x=0},
\ee
see, e.g., \cite{Hormander3}. 
In our case, $p(x,D)=(4\pi)^{-1}|D|\phi(x)$, i.e., it has an amplitude $a(x,y,\xi)= (4\pi)^{-1}|\xi|\phi(y)$. For its symbol $p(x,\xi)$ we have 
\[
4\pi  p(x,\xi)=    \phi(x)|\xi| -\i  \phi'(x)\sgn\xi  .
\]
The symbol of the Heaviside function $h_0$ is $-\i/\xi$ (away from $\xi=0$), therefore, in \r{L3a}, we have
\be{L3aa}
4\pi p(x,\xi)a(\xi) =  -\i \phi(x)\sgn\xi - \phi'(x)\sgn\xi /\xi .
\ee
By \r{L3a}, the symbol of the conormal distribution $4\pi \phi h_0$ then is
\be{L3b}
4\pi p(x,\xi)a(\xi) =  -\i \phi(0)\sgn\xi - \phi'(0)\sgn\xi /\xi + \i\phi''(0) \sgn\xi/\xi^2 \mod S^{-3}.
\ee
Now, $-\i\sgn\xi$ is the symbol (the Fourier transform) of $\pi^{-1}\pv (1/x)$. Next, $-\sgn\xi/\xi$ is the symbol of the distribution with derivative having symbol  $-\i \sgn\xi$, which is $\pi^{-1}\pv (1/x)$. Taking antiderivative of the latter, we get $\pi^{-1}\log|x|$. Finally, $\i \sgn\xi/\xi^2$ is obtained from $-\i\sgn\xi$ by multiplying by $-1/\xi^2$, which corresponds to taking the second antiderivative; hence we get $\pi^{-1}x(\log|x|-1)$, which is a continuous function. The latter also follows from the fact that a symbol $\sim 1/\xi^2$ at $\xi\to\infty$ is $L^1$ there, therefore, its inverse Fourier transform is continuous. By the same argument, the remainder in \r{L3b} produces a $C^1$ function, and one can get a complete singular expansion, in fact. This proves \r{L1} in the lemma. 

Equation~\r{L2} follows in a similar way. The symbol of the conormal distribution $a(\xi) = x_+$ is $-1/\xi^2$, therefore in \r{L3aa} we have 
\[
4\pi p(x,\xi)a(\xi) =  -\phi(x)/|\xi| +\i \phi'(x)\sgn\xi /\xi^2 .
\]
instead. The second term produces a continuous function while the first one, by the calculations, above, would produce a leading term $\pi^{-1}\phi(0)\log|x|$ plus another continuous function.  

For the last identity \r{L3} in the lemma, we need the symbol of $x_+^{1/2}$. Since 
$\mathcal{H}= (4\pi)^{-1} |D| =(4\pi)^{-1} H\, \d/\d x $, dropping the factor $(4\pi)^{-1}$ for a while, we can apply $\d/\d x$ first to study $x_+^{-1/2}$. 
We have 
\[
(x_+^{\lambda-1}) \hat{\ } = \Gamma(\lambda)\left( e^{-\i\pi\lambda/2} \xi_+^{-\lambda} + e^{\i\pi\lambda/2} \xi_-^{-\lambda}\right), \quad \lambda\not\in\mathbb{Z}, 
\]
see, e.g., \cite[Ch.~8.6]{Friedlander1998}.  When $\Re\lambda>0$, $x_+^{\lambda-1}$ is locally integrable and in general, it is defined by analytic extension in $\lambda$. 
Therefore, with $\lambda=1/2$, 
\[
(x_+^{-1/2})\hat{\ } = \sqrt\pi \left( e^{-\i\pi/4} \xi_+^{-1/2} + e^{\i\pi/4} \xi_-^{-1/2}\right).
\]
Then $H$ is a multiplication with $-\i\sgn(\xi)$ 
on the Fourier side, which happens to make sense on $ x_+^{-1/2}$, hence
\[
\begin{split}
(H x_+^{-1/2})\hat{\ } &= -\i \sqrt\pi \left( e^{-\i\pi/4} \xi_+^{-1/2} - e^{\i\pi/4} \xi_-^{-1/2}\right)\\
&=  - \sqrt\pi \left( e^{\i\pi/4} \xi_+^{-1/2} - e^{\i3\pi/4} \xi_-^{-1/2}\right)\\
& = - \sqrt\pi \left( e^{\i\pi/4} \xi_+^{-1/2} + e^{-\i\pi/4} \xi_-^{-1/2}\right) = - (x_-^{-1/2})\hat{\ } .
\end{split}
\]
Therefore, $H x_+^{-1/2} = - x_-^{-1/2}$.

We can use this calculation in \r{L3a}, where $a(\xi)= (x_+^{1/2})\hat{\ } $ to prove \r{L3}. The next (continuous) term as a square root  singularity, which we will not investigate. 
\end{proof}

\section{Recovery of edges, an asymptotic view} 
\subsection{The direct method} 
Recovery of edges will be analyzed here based on Theorems~\ref{thm_1} and \ref{thm_art}. 
In the numerical examples in the previous section, we can see that besides the aliasing artifacts creating specific singularities along the edge, the actual jump looks well recovered. The horizontal profiles there are a Gaussian cut by half by the Heaviside function, which creates a jump of size one. In Figure~\ref{fig_flat-edge}(c), one can see a jump one with $C\pv(1/x)$ added. In Figure~\ref{fig_convex-edge}(c), if we average the $x_-^{-1/2}$ oscillations on the left, the jump is still close to one. Finally, in Figure~\ref{fig_corner-edge}(c), the (weaker) $\log|x|$ singularity is added to a smoothened out cut-off Gaussian with a jump close to one, as well. As explained in that section, removing the predicted singularities, what remains is a continuous function, so the jumps are smoothened out. The reason they appear close to actual jumps in those numerical examples is that the angular step size is not ``too small'' but it is still ''small.'' If we increase it, the jump  do not look well recovered anymore.

The observed effect is better understood, in author's view, asymptotically, as the angular step size tends to zero. As explained earlier, we assume now that $f=f_h$ is a \sc ly band limited function with bound $B$. 

We start with a general observation which we will not formalize as a theorem. Consider a jump type of singularity. Locally, after a change of variables,  it is a multiple of the Heaviside function $h_0(x_2)$ in the $x_2$ variable, modulo lower order terms. To account for the localization, we represent it as  $f= \phi(x)h_0(x_2)$ with some $\phi\in C_0^\infty$. It is convenient to assume that $\phi(x) = \phi_1(x_1) \phi_2(x_2)$. Then 
\[
 \hat f(\xi)  =  \hat\phi_1(\xi_1) \hat \phi_2*  \Big(\pi \delta(\xi_2)-\i. \text{pv}\frac1{\xi_2}\Big)  = \hat\phi_1(\xi_1)\Big(\pi \hat \phi_2 (\xi_2)   -\i\hat\phi_2 *\text{pv}\frac1{\xi_2}\Big) .
 \]
Assuming $f$ smoothened by a convolution with some $\psi_h$ as above, we get 
\be{F2}
\mathcal{F}_h (\psi_h* f)(\xi)  = \hat  \psi(\xi) \hat f(\xi/h) = -\i \hat  \psi(\xi)\hat \phi_1(\xi_1/h) \Big[  \hat \phi_2 (\cdot)  *\text{pv}\frac1{\cdot}\Big](\xi_2/h) + O(h^\infty). 
\ee
The only rays along which this is not $O(h^\infty)$ in a conic neighborhood are the ones parallel to the $\xi_2$ direction. Along them, $\xi_1=0$ and the expression in the brackets has the asymptotic $\sim h/\xi_2$ for $|\xi_2|>1/C$, $\forall C$. With this in mind, \r{F2} is like $-\i h \hat\psi(0,\xi_2)\hat \phi_1(0)\frac1{\xi_2}$  along the axis $\xi_1=0$, which matters the most. The factor $\hat\psi(0,\xi_2)$ plays a role of a low pass filter modeling the effect of averaging the measurements. If its cutoff frequency, call it $B$, satisfies \r{F1}, then there is no aliasing. When it does not, and this is the case we want to understand, there is aliasing as explained earlier. We get artifacts along the line tangent to the curve where the jump occurs, passing to a point where it happens. It is important to note that the number of non-negligible terms in \r{thm_1_eq1}, restricted to $\mathcal{B}(0,R)$ is independent of $h$ and depends on $B$ only. 
 The factor $h$ above shows that the aliasing artifacts decrease as $h$ when $h\to0$.


\textbf{Numerical example.} We take a function jumping from $0$ to $1$ in a slightly smoothened way, across the parabola $x=1.5y^2$ in the plane. Instead of taking $f=h_0(x-1.5y^2)$, we replace the Heaviside function by $h_\lambda(t)= \frac12(1+\erf(\lambda t))$, where $\erf$ is the ``error function'' defined as the normalized antiderivative of the Gaussian $e^{-t^2}$ with $\erf(0)=0$, and $\lim_{t\to\pm\infty}\erf(t)=\pm1$.  Then $h_\lambda$ is the Heaviside function convolved with a highly concentrated Gaussian, as $\lambda\gg1$. Its Fourier transform multiplies $1/(\i\xi)$ (for $\xi\not=0$) by a Gaussian as well. While that multiplier is not compactly supported, for all computational purposes here, it is. This makes the jump function \sc ly band limited with $h$ proportional to $1/\lambda$. Finally, we localize $f=h_\lambda(x-1.5y^2)$ by multiplying by a function of compact support equal to one near the vertex. 

\begin{figure}[h!] \hspace{0.04\textwidth} \hfill 
\centering
\begin{subfigure}[t]{.25\textwidth}
\centering
\includegraphics[height=.172\textheight]{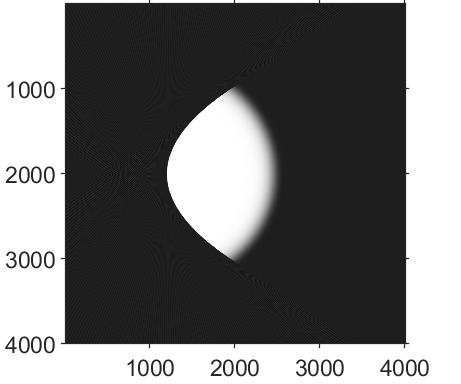}
\caption{A convex edge}
\end{subfigure}
\hfill
\begin{subfigure}[t]{.7\linewidth}
\centering
\includegraphics[height=.19\textheight]{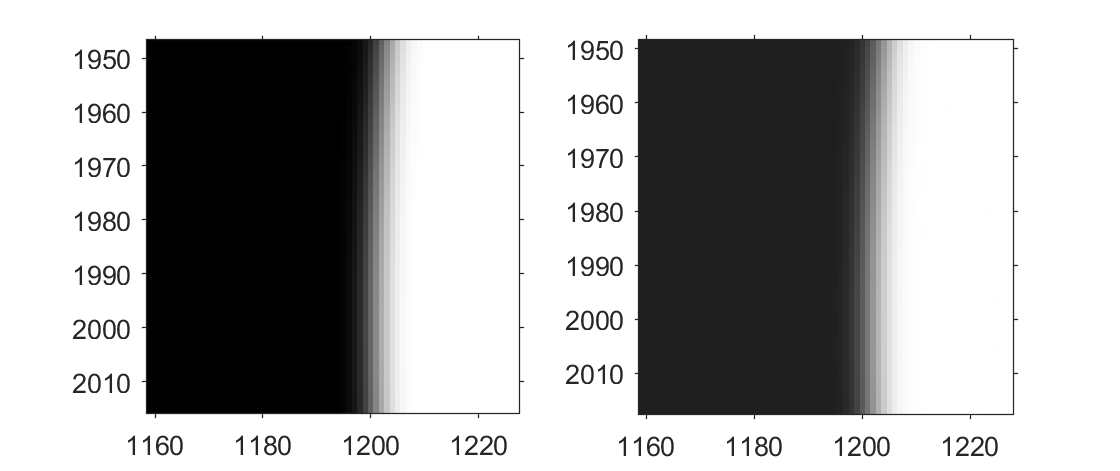}
\caption{Highly zoomed in: original vs.\ recovered, $s=0.8^\circ$.}
\end{subfigure}\\  \vspace{0.02\textheight}
\hfill 
\begin{subfigure}[t]{.38\linewidth}
\centering
\includegraphics[height=.15\textheight]{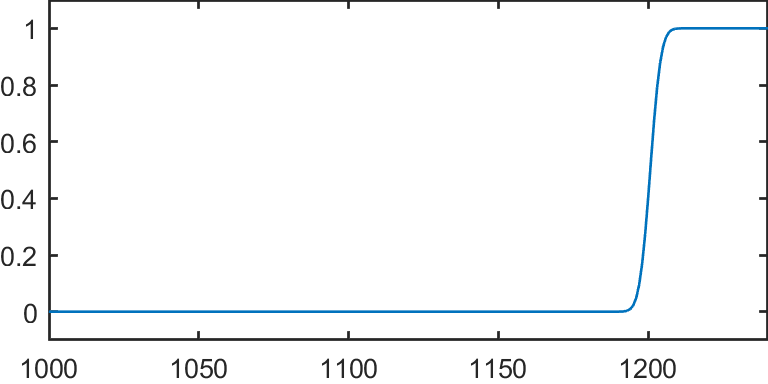}
\caption{A horizontal cross-section plot of the phantom at the vertex over a small part of the horizontal axis.}
\end{subfigure}\ \ \ \ \hfill  \ \ \ \ 
\begin{subfigure}[t]{.38\linewidth}
\centering
\includegraphics[height=.15\textheight]{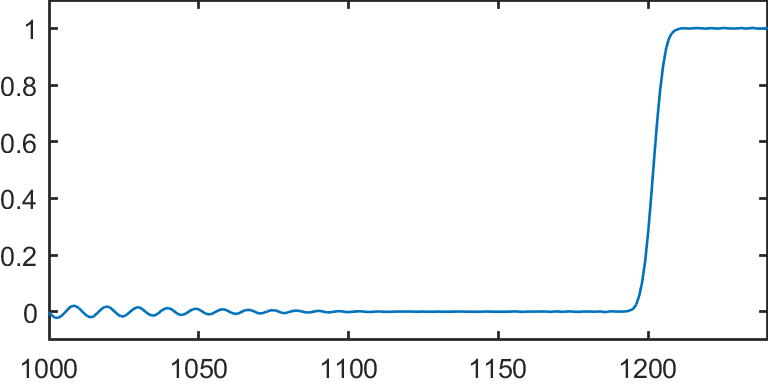}
\caption{A horizontal cross-section plot of the reconstruction at the vertex..}
\end{subfigure}\hfill\hfill{}
\caption{A convex ``\sc\ edge'' with $s=0.8^\circ$. 
}\label{fig_convexSC-edge}
\end{figure}

We take $\lambda=500$. We perform the computations on a $4,000\times 4,000$ grid with an angular step of $0.8^\circ$. The phantom is shown in Figure~\ref{fig_convexSC-edge}(a). The reconstructed one looks virtually the same with the artifacts barely visible,  shown in Figure~\ref{fig_convex_SC-artifacts} on a different scale. We zoom in at the vertex of the hyperbola in Figure~\ref{fig_convexSC-edge}(b) to compare the original phantom and the reconstruction. The squares shown are approximately $70\times 70$ pixel crops of the $4,000\times 4,000$ original and of the $4,002\times 4,002$ recovery, respectively (MATLAB's {\tt iradon} adds a pixel on each side if the output size is not specified). In Figure~\ref{fig_convexSC-edge}(c)(d), we show plots of horizontal cross-sections of the edge through the vertex, well stretched compared to (a), with  6\% of the total cross-section plotted. The edge is very will recovered, and the artifacts (the low amplitude oscillations) are separated from the edge at a distance controlled by the effective band limit of $f$. 

Finally, in Figure~\ref{fig_convex_SC-artifacts}, we show an approximately $525\times 525$ crop, zoomed in, of the vertex area rendered to the range of values $[-0.1, 0.1]$ (the original one is $[0,1]$) to emphasize on the artifacts.  We see that in some neighborhood of the edge, there are no artifacts. This is consistent with the cross-section plot in Figure~\ref{fig_convexSC-edge}(c), and with Figure~\ref{fig_2}, right, see also Remark~\ref{rem_4.2}. 

\begin{figure}[h!] 
  \centering
	\includegraphics[ scale=0.3]{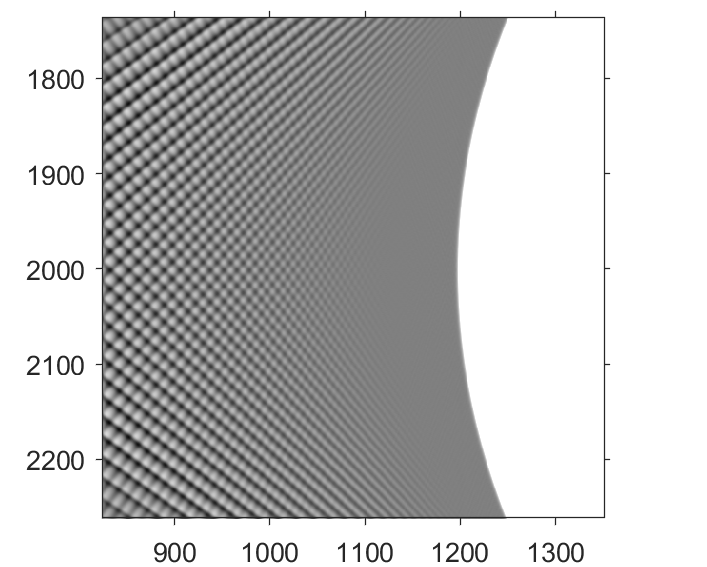}
\caption{\small A zoomed in $525\times 525$ crop of the reconstruction in Figure~\ref{fig_convexSC-edge}, with range $[-0.1,0.1]$, clipping all values $\ge0.1$. A neighborhood of the edge is free of artifacts.}
\label{fig_convex_SC-artifacts}
\end{figure}

Compared to the situation on section~\ref{sec-classical}, we have an artifact free neighborhood (in the case of convex edges), and the \sc\ singularities not extending too far from a point. Also, those are \sc\ singularities, high oscillations instead of being classical one. 

\subsection{The interpolation method} 
%
%

We comment briefly on the recovery of edges with the interpolation method using Theorem~\ref{thm_2}. 
In Figure~\ref{fig_disk}, we present a numerical example with a characteristic function, slightly blurred, of a disk placed off center. The conversion in (b) is the direct one. The artifacts are separated from the edge and extend everywhere. 
\begin{figure}[h!]\,\hfill
\begin{subfigure}[t]{.2\linewidth}
\centering
\includegraphics[height=.15\textheight]{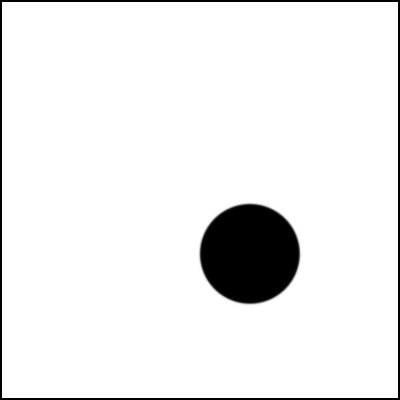}
\caption{The original $f$}
\end{subfigure}\hfill
\begin{subfigure}[t]{.2\linewidth}
\includegraphics[height=.15\textheight]{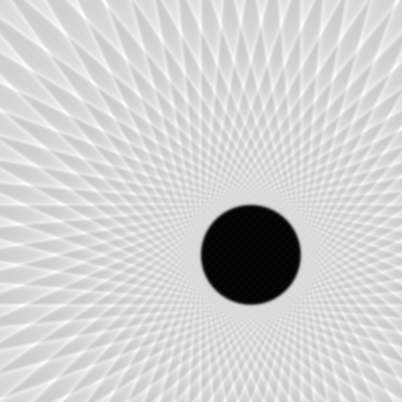}
\caption{Direct inversion with a $5^\circ$ step}
\end{subfigure}
\begin{subfigure}[t]{.2\linewidth}
\end{subfigure}\hfill
\begin{subfigure}[t]{.2\linewidth}
\includegraphics[height=.15\textheight]{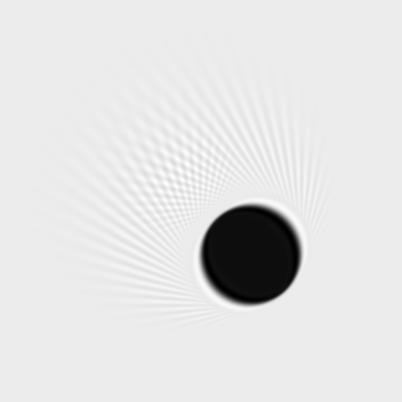}
\caption{Interpolated inversion with a $5^\circ$ step}
\end{subfigure}\hfill
\begin{subfigure}[t]{.2\linewidth}
\includegraphics[height=.15\textheight]{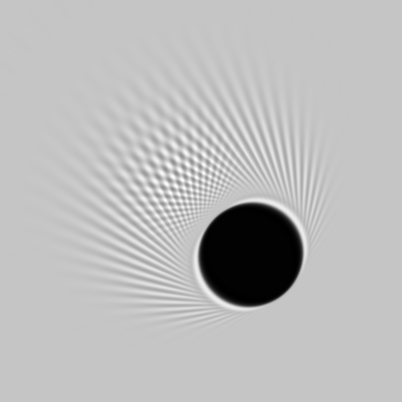}
\caption{As (c) with contrast enhanced}
\end{subfigure}\hfill {}
\caption{A function $f$ with a jump type singularity and its recovered version from $\mathcal{R}f$ sampled with a $5^\circ$ step (top row) and a $3^\circ$ step (bottom row). The origin is in the center. 
}\label{fig_disk}
\end{figure}
The reconstruction in (c) is the interpolated one, and the aliasing artifacts are localized in accordance with Theorem~\ref{thm_2} and Figure~\ref{fig_1}.   In (c), we plot a version with an enhanced contrast. 

Note that in (c) and in (d), parts of the edge are more blurred than the original, and some oscillations (Gibb's like effect) are visible. The explanation is that when an artifact is created, that frequency is removed from the edge, since for each of them, only one $k$ is possible in \r{CR}, \r{RA2}. We used the Lanczos-3 interpolation here, which has an oscillating kernel. This, and Theorem~\ref{thm_conv} explain why those edges have oscillations when reconstructed. The effect is stronger for the edges with tangents passing through the origin since they would be most affected by the angular convolution.

\appendix
\section{Semiclassical sampling} \label{sec_appendix}
We summarize some of the results in \cite{S-Sampling}. 

\subsection{Elements of \sc\ analysis} 
Our reference for \sc\ analysis is \cite{Zworski_book}. We consider functions $f_h(x)$ depending on $x\in\R^n$ and on a small parameter $h>0$ as well. Often, we suppress the dependence on $h$ and just write $f_h=f$. 
The \sc\ Fourier transform $\mathcal{F}_hf(\xi) = \hat f(\xi/h)$ is just a rescaled classical one. 
We restrict our attention here to functions called \textit{localized in phase space} in \cite{Zworski_book}, and \textit{\sc ly band-limited} in \cite{S-Sampling}. Those are functions $f_h$ with the following properties. Each one is (i) supported in an $h$-independent compact set, (ii) is tempered (the $H^s$ norm is polynomially bounded in $h^{-s}$ for some $s$), and (iii) there exists a compact set $\mathbf{B}\subset \R^n$, so that for every open $U\supset \mathbf{B}$, we have $|\mathcal{F}_hf(\xi) | \le C_Nh^N \langle\xi\rangle^{-N}$, $\forall N$. Then we say that the set $\mathbf{B}$ is the band limit of $f$.  Depending on the application,  when the latter is the ball $\mathcal{B}(0,B)$ for some $B>0$, then $B$ is called a band limit or when it is the square $[-B, B]^n$, then $B$ is the band limit. Note that the notion of band limit depends on the coordinate system but then the sampling geometry does as well. 

Such functions belong to $C_0^\infty(\R^n)$ for every $h$ but they can oscillate highly when $h\to0$. The \sc\ wave front set $\WFH(f)$ is the set of points $x$ and co-directions $\xi$ defined as the complement of those $(x_0,\xi_0)$ for which there is $\phi\in C_0^\infty(\R^n)$ with $\phi(x_0)\not=0$ so that $\mathcal{F}_h(\phi f) = O(h^\infty)$. We call the points $(x,\xi)\in \WFH(f)$ \sc\ singularities. The \sc\ wave front set is not conic in general. The projection of $\WFH(f)$ onto the dual variable $\xi$ is called the frequency set $\Sigma_h(f)$. It is, in fact, the smallest band limit $\mathbf{B}$. 

Semiclassical \PDO s are defined as
\be{hPDO}
Pf(x) = (2\pi h)^{-n}\iint e^{\i (x-y)\cdot\xi/h} p(x,\xi) f(y)\,\d y\, \d\xi,
\ee
where, for every compact set $K$ and $\alpha$, $\beta$,  the symbol $p(x,\xi)$, possibly depending on $h$ as well, satisfies
\be{hpdo1}
|\partial_x^\alpha \partial_\xi^\beta p(x,\xi)|\le C_{K,\alpha,\beta}h^k\langle \xi\rangle^{m-|\beta|}
\ee
for some $k$ and $m$. 
Acting on \sc ly band limited functions with a fixed band limit, one can just take a compactly supported $p$, so the decay in $\xi$ above would be automatic. (Local) semiclassical Fourier Integral Operators (FIOs) are defined similarly but with a phase function $\phi(x,y,\xi)$ satisfying some conditions, see \cite{Guillemin-SC}, \cite{Martinez_book}.

\subsection{Semiclassical sampling} The \sc\ sampling theory developed in \cite{S-Sampling} is an asymptotic version of the classical one. For a \sc ly limited function with $\Sigma_h(f)\subset [-B_1,B_1]\times\dots\times [-B_n,B_n]$, it is enough to know its samples (which number is $O(h^{-n})$) on a uniform rectangular grid  of side  $s_jh_j$ in each direction, with $s_j<\pi/B_j$, in order to recover $f$ up to an $O(h^\infty)$ error. The reconstruction formula is of interpolation type \r{P1}, 
so that $\chi(x)= \chi_1(x_1)\dots \chi_n(x_n)$,  $\hat\chi_j\subset (-\pi,\pi)^n$, and $\hat \chi_j(\pi \xi_j/B_j)=1$ for $\xi\in \Sigma_h(f) $, under the condition $0<s_j<\pi/B_j$. 

If $A$ is an FIO (classical), and $f$ is as above, one can determine the smallest box where $\WFH(Af)$ is contained by studying the canonical relation of $A$. In particular, this applies to $\RR$ and to its inverse $\RR^{-1}$. This allows us to compare the sharp sampling requirements for $f$ and $Af$ (and for $A^{-1}f$ if $A$ is elliptic, associated to a local diffeomorphism, like $\RR$). 


\subsection{Aliasing} 
If the Nyquist condition $s_j <\pi/B_j$ is not satisfied, aliasing occurs. For simplicity, assume all $B_j$ equal (can be done by a linear transformation). 
As in the classical case, frequencies ``fold'' over the Nyquist box. The interpolation formula approximates not $f_h$ but 
\be{A2}
Gf := \mathcal{F}_h^{-1} \hat\chi(s\cdot)\sum_{k\in\mathbb{Z}^n}  \mathcal{F}_h f_h(\cdot+2\pi k/s). 
\ee
When there is a non-trivial contribution from $k\not=0$, we get aliasing artifacts.

Writing $ G=\sum_{k\in\mathbf{Z}} G_k$, we get that each $G_k$ is an h-FIO with a canonical relation given by the shifts 
\be{sh}
S_k:(x,\xi) \longmapsto    (x,\xi+{2\pi} k/s).
\ee
This FIO preserves the space localization (as it is clear from \r{sh}) but shifts the frequencies, which can be viewed as changing the direction and the magnitudes of the latter. We identify in this paper canonical relations with the maps they induce. 

Assume now that $A$ is elliptic, associated to a local diffeomorphism $C$, like $\RR$. Assume that the measurement  $Af$ is aliased, and we apply the parametrix $A^{-1}$.  Then the inversion would be $A^{-1}G_kA$; and by the h-FIO calculus, away from zero frequencies, that is an h-FIO with a canonical relation  $C^{-1}\circ S_k\circ C$ acting on $(x,\xi)\in\supp\hat \chi(s\cdot+2k\pi) $. The classical aliasing   creates artifacts at the same location but with shifted frequencies. The artifacts here however could move to different locations, as it happens for the Radon transform.

\subsection{Sampling on the unit circle} 
The circle is a manifold, with no unique chart possible (but two suffice). The definition of a band limit is not invariant under coordinate changes but is invariant under rigid motions, so it requires some clarification what it means on the unit circle. 

Let $f$ be a function on the unit circle. 
We can think of it as a function of the polar angle $\varphi$, periodic with period $2\pi$. The natural Fourier transform is an expansion in Fourier series. 
On the other hand, there are natural coordinate maps on the unit circle preserving the arc-length. We can remove any fixed point  $x_0$ from it, say having a polar angle  $\varphi_0 \mod 2\pi$ and map the rest to $(\varphi_0,\varphi_0+2\pi)$ by the polar angle. 
Given a distribution  $f$ on $S^1$, depending on $h$, we can localize it to that chart by a smooth cut-off $\chi$. 

\begin{definition}\label{def_circle1}
If $\chi f$ is \sc ly band limited for every such chart, we call $f$ \sc ly band limited with band limit $B$ being the supremum of the band limit over all such charts.
\end{definition}

\begin{lemma}\label{lemma_polar}
The supremum $B$ in Definition~\ref{def_circle1} is finite. Moreover, $B= \max(B_1,B_2)$, where $B_1$, $B_2$ are two such band limits for two charts corresponding to two distinct cut-off  points, and the corresponding $\chi_1$, $\chi_2$ form a partition of unity. 
\end{lemma}
\begin{proof}
For every distribution on $S^1$, we can write $f=\chi_1 f+ \chi_2 f$. Let $x_0\in S^1$ with a polar angle $\varphi_0\mod 2\pi$ be a cut-off point for a local chart. Let $\chi_0\in C^\infty(S^1)$ be zero near $x_0$. Then $\chi_0 f = \chi_0\chi_1 f+ \chi_0\chi_2 f$. The term $\chi_0\chi_1 f$ can be written as a sum of two functions: one supported between $x_0$ and $x_1$ (going in positive direction along the circle), and the other one supported between $x_1$ and $x_0$. They both can be re-mapped to the chart associated with $x_1$ at the expense of possible shifting by $2\pi k$, $k\in\mathbb{N}$. That shift does not change the \sc\ band limit, and a multiplication by a $C_0^\infty$ function cannot make it greater; therefore, the \sc\ band limit of $\chi_0\chi_1 f$ does not exceed $B_1$. We analyze  $\chi_0\chi_2 f$ in the same way to get an upper bound $B_2$. Therefore, an upper bound is  $B= \max(B_1,B_2)$ but since it is attained for either $\chi_1 f$ or $\chi_2 f$, it is actually the least one. \mar{explain}
\end{proof}

\begin{definition}\label{def_circle2}
The function $f_h\in C^\infty(S^1)$ is called \sc ly band limited with band limit $B$, if (i) it is tempered, i.e., $\|f_h\|_{L^2(S^1)}\le Ch^{-N}$ for some $N$, (ii)  and for its Fourier coefficients $f_n$, for each $B'>B$,  we have 
\be{A10}
|f_n|\le C_N   |n|^{-N} , \quad |n|>B'/h.
\ee
\end{definition}

\begin{proposition}
Definition~\ref{def_circle1} and Definition~\ref{def_circle2} are equivalent. 
\end{proposition}
\begin{proof}
Let $f$ be a \sc ly band limited with a band limit $B$, according to Definition~\ref{def_circle1}. Since $\chi f$ is tempered for any cutoff $\chi$ as in Definition~\ref{def_circle1}, we deduce that $f$ is tempered, too. 
The Fourier coefficients of $f$ are given by
\[
f_n = \int_0^{2\pi} e^{-\i n\varphi}f(\omega(\varphi))\,\d\varphi.
\]
We view the integration above as an integration over $S^1$ since $e^{-\i n\varphi}$ is $2\pi$-periodic. Then we apply the partition of unity $1=\chi_1+\chi_2$ to $f$ as in Lemma~\ref{lemma_polar}. The integral of each term resulting from that can be written as an integral over a subinterval of the real line. It is enough to consider the first one only. We have
\[
\int_{\varphi_1}^{\varphi_1+ 2\pi} e^{-\i \varphi (hn)/h}(\chi_1f)(\omega(\varphi))\,\d\varphi.
\]
This is the \sc\ Fourier transform of $\chi_1 f$ evaluated at $\hat\varphi=hn$. It is $O(h^N\langle \hat\varphi\rangle^{-N})$ for every $N$, and for $|\hat\varphi|>B'>B$, which implies $O(h^N(1+h|n|)^{-N})$, hence  $O(1+|n|)^{-N}$ for $|n|>B'/h$. 

Assume Definition~\ref{def_circle2} now. Then $f$ is tempered and we have \r{A10}. Write
\[
f(\varphi) =\frac1{2\pi} \sum e^{\i n\varphi} f_n.
\]
For $\chi\in C_0^\infty$, 
\be{A11}
\mathcal{F}_h\chi f(\hat \varphi) = \frac1{2\pi} \sum_{n=-\infty}^\infty   f_n \hat\chi(\hat \varphi/h-n).
\ee
Choose $B'>B''>B$ and restrict $\hat \varphi$ to $|\hat \varphi|>B'$.  Notice first that 
\be{A11a}
|\hat\chi(\hat \varphi/h-n)|\le C_N| \hat \varphi/h-n |^{-N} = C_N h^N |\hat \varphi- hn|^{-N}.
\ee
Summing over $|n|\le B''/h$ in \r{A11}, we get
\be{A12}
\Big| \sum_{|n|\le B''/h}  f_n \hat\chi(\hat \varphi/h-n)\Big| \le C_N'h^N|\hat\varphi |^{-N} \quad  \text{for $|\hat \varphi|>B'$},
\ee
where we used \r{A11a}, and the fact that the number of terms above is $O(h^{-1})$. For the remainder of the sum, we have
\be{A13}
\Big| \sum_{|n|> B''/h}  f_n \hat\chi(\hat \varphi/h-n)\Big| \le C_N \sum_{|n|> B''/h}  |n|^{-N-2} |\hat\chi(\hat \varphi/h-n)|. 
\ee
We want to show that it is $O((h/|\varphi|)^N)$, $\forall N$. We will multiply by $(\hat \varphi/h)^N = \big((\hat \varphi/h-n) +n\big)^N$ and show that it is uniformly bounded.. Using the binomial formula, we just need to show that multiplying \r{A13} by $n^{N-k}(\hat \varphi/h-n)^k$, $0\le k\le N$, leaves it uniformly bounded. Since $\hat\chi$ is Schwartz class, it is enough to estimate
\[
\sum_{|n|> B''/h} | n|^{-N-2}|n|^{N-k}  = \sum_{|n|> B''/h}  |n|^{-2} |n|^{-k} \le C. 
\]
Therefore, \r{A13} is $O((h/|\varphi|)^N)$, indeed. This, combined with \r{A12} shows the same for $\mathcal{F}_j\chi f$ for $|\hat\varphi|>B'$, for every fixed $\chi\in C_0^\infty$.
\end{proof}

Finally, we will mention that on the circle, the sinc interpolation of classically band-limited functions on it (trigonometric polynomials) has its analog as well, see \cite{Epstein-book}.


\end{document}